\documentclass[twoside]{article}

\usepackage[accepted]{aistats2023}
\usepackage[round]{natbib}
\usepackage{adjustbox}

\setcitestyle{authoryear}

\usepackage[utf8]{inputenc} %
\usepackage[T1]{fontenc}    %
\usepackage{hyperref}       %
\usepackage{url}            %
\usepackage{booktabs}       %
\usepackage{amsfonts}       %
\usepackage{nicefrac}       %
\usepackage{microtype}      %
\usepackage{xcolor}         %
\usepackage{graphicx,wrapfig}

\usepackage{amssymb,amsmath,amsthm}
\usepackage{multirow,multicol}
\usepackage{algorithm,algorithmic}

\usepackage{subfig}

\usepackage[shortlabels]{enumitem}
\newtheorem{theorem}{Theorem}
\newtheorem{lemma}{Lemma}
\newtheorem{proposition}{Proposition}
\newtheorem{assumption}{Assumption}
\newtheorem{definition}{Definition}
\newtheorem{corollary}{Corollary}
\theoremstyle{remark}
\newtheorem{example}{Example}

\usepackage{pifont}

\def\grad{\nabla}

\def\ba{\mathbf{a}}
\def\bb{\mathbf{b}}

\def\bd{\mathbf{d}}

\def\bn{\mathbf{n}}

\def\bs{\mathbf{s}}

\def\bw{\mathbf{w}}
\def\bx{\mathbf{x}}  %
\def\by{\mathbf{y}}
\def\bz{\mathbf{z}}
\def\bA{\mathbf{A}}
\def\bB{\mathbf{B}}

\def\bD{\mathbf{D}}

\def\bX{\mathbf{X}}

\def\cD{\mathcal{D}}

\def\cH{\mathcal{H}}

\def\cN{\mathcal{N}}
\def\cO{\mathcal{O}}

\def\cR{\mathcal{R}}

\def\cX{\mathcal{X}}

\def\cZ{\mathcal{Z}}

\def\smskip{\smallskip}

\def\texitem#1{\par\smskip\noindent\hangindent 25pt
               \hbox to 25pt {\hss #1 ~}\ignorespaces}

\def\abs#1{\left|#1\right|}
\def\norm#1{\left\|#1\right\|}

\def\fprod#1{\left\langle#1\right\rangle}

\newcommand{\BEAS}{\begin{eqnarray*}}
\newcommand{\EEAS}{\end{eqnarray*}}
\newcommand{\BEA}{\begin{eqnarray}}
\newcommand{\EEA}{\end{eqnarray}}
\newcommand{\BEQ}{\begin{eqnarray}}
\newcommand{\EEQ}{\end{eqnarray}}
\newcommand{\BIT}{\begin{itemize}}
\newcommand{\EIT}{\end{itemize}}
\newcommand{\BNUM}{\begin{enumerate}}
\newcommand{\ENUM}{\end{enumerate}}

\newcommand{\BA}{\begin{array}}
\newcommand{\EA}{\end{array}}

\newcommand{\reals}{\mathbb{R}}

\def\Pr{\mathbb{P}}

\newcommand{\dist}{\mathop{\rm dist{}}}

\DeclareMathOperator*{\argmin}{\arg\!\min}

\def\blue#1{\textcolor{blue}{#1}}

\newif\ifpagenumbering
\pagenumberingtrue

\pagenumberingfalse

\newsavebox{\theorembox}
\newsavebox{\lemmabox}
\newsavebox{\defnbox}
\newsavebox{\corollarybox}
\newsavebox{\remarkbox}
\newsavebox{\assbox}
\savebox{\theorembox}{\noindent\bf Theorem}
\savebox{\lemmabox}{\noindent\bf Lemma}
\savebox{\defnbox}{\noindent\bf Definition}
\savebox{\corollarybox}{\noindent\bf Corollary}
\savebox{\remarkbox}{\noindent\bf Remark}
\newtheorem{remark}{\usebox{\remarkbox}}[section]

\newcommand{\FW}{\mathcal{G}}
\newcommand{\bigO}{\mathcal{O}}
\newcommand{\underf}{\underline{f}}

\def\bbeta{\boldsymbol{\beta}}
\def\blambda{\boldsymbol{\lambda}}

\newcommand{\rj}[1]{\textcolor{blue}{#1}}
\renewcommand{\blue}{}
\renewcommand{\rj}{}
\begin{document}

\runningtitle{A Conditional Gradient-based Method for Simple Bilevel Optimization with Convex Lower-level Problem}

\runningauthor{Ruichen Jiang, Nazanin Abolfazli, Aryan Mokhtari, Erfan Yazdandoost Hamedani}

\twocolumn[

\aistatstitle{A Conditional Gradient-based Method for Simple Bilevel Optimization\\ with Convex Lower-level Problem}

\aistatsauthor{Ruichen Jiang \And \hspace{-2em}Nazanin Abolfazli \And  \hspace{-2em}Aryan Mokhtari \And \hspace{-2em}Erfan Yazdandoost Hamedani}

\aistatsaddress{ UT Austin \And  \hspace{-2em}University of Arizona \And \hspace{-2em}UT Austin \And \hspace{-2em}University of Arizona} ]

\begin{abstract}
In this paper, we study a class of bilevel optimization problems, also known as simple bilevel optimization, where we minimize a smooth objective function over the optimal solution set of another convex constrained optimization problem. Several iterative methods have been developed for tackling this class of problems. Alas, their convergence guarantees are either asymptotic for the upper-level objective, or the convergence rates are slow and sub-optimal. To address this issue,
in this paper, we introduce a novel bilevel optimization method that locally approximates the solution set of the lower-level problem via a cutting plane and then runs a conditional gradient update to decrease the upper-level objective.  When the upper-level objective is convex, we show that our method requires ${\mathcal{O}}(\max\{1/\epsilon_f,1/\epsilon_g\})$ iterations to find a solution that is $\epsilon_f$-optimal for the upper-level objective and  $\epsilon_g$-optimal for the lower-level objective. Moreover, when the upper-level objective is non-convex, our method requires ${\mathcal{O}}(\max\{1/\epsilon_f^2,1/(\epsilon_f\epsilon_g)\})$ iterations to find an $(\epsilon_f,\epsilon_g)$-optimal solution.
We also prove stronger convergence guarantees under the H\"olderian error bound assumption on the lower-level problem.
To the best of our knowledge, our method achieves the best-known iteration complexity for the considered class of bilevel problems. %
\end{abstract}

\section{INTRODUCTION}\label{sec:intro}
Bilevel optimization is a form of optimization where one problem is embedded within another. It captures
a hierarchical structure, where an \emph{upper-level} function is minimized over the solution set of a \emph{lower-level} problem. 
This class of problems has attracted great attention due to their applications in
hyper-parameter optimization \citep{franceschi2018bilevel,shaban2019truncated}, meta-learning \citep{rajeswaran2019meta,bertinetto2018metalearning}, and reinforcement learning \citep{hong2020two}.
In this paper, we focus on a specific form of bilevel optimization formally defined as 
\begin{equation}\label{eq:bi-simp}
    \min_{\bx\in \reals^d}~f(\bx)\quad \hbox{s.t.}\quad  \bx\in\argmin_{\bz\in \cZ}~g(\bz),
\end{equation} 
where $\cZ$ is a compact convex set and %
$f,g:\reals^d\to \reals$ are continuously differentiable functions on an open set containing $\cZ$. We assume that $g$ is convex but not necessarily strongly convex, and hence the lower-level problem in \eqref{eq:bi-simp} could have multiple optimal solutions.  
We remark that problem~\eqref{eq:bi-simp} is often referred to as the ``simple bilevel problem'' in the literature \citep{dempe2010optimality,dutta2020algorithms,shehu2021inertial} to differentiate it from the more general settings where the lower-level problem is parameterized by some upper-level variables. This class of bilevel problems appears in several settings as discussed in Section~\ref{sec:pre}. The key challenge to solve problem~\eqref{eq:bi-simp} stems from the fact that its feasible set---the solution set of the lower-level problem---does not admit a simple characterization and is not explicitly given. This rules out the possibility of directly applying projection-based methods as well as the conditional gradient (CG) type methods, since projection onto or minimizing a linear objective over the feasible set is intractable.

\begin{table*}[t!]\scriptsize
     \renewcommand{\arraystretch}{1.0}
    \centering
    \caption{Summary of bilevel optimization algorithms. The abbreviations ``SC'', ``C'', and ``non-C'' stand for ``strongly convex'', ``convex'', and ``non-convex'', respectively.}\label{tab:bilevel}
     \resizebox{\textwidth}{!}{%
        \begin{tabular}{ccccccc}
            \toprule
			\multirow{2}{*}{References}			    & Upper level 		                                            & \multicolumn{2}{c}{Lower level}                                       & \multicolumn{2}{c}{Convergence}					    & \multirow{2}{*}{Oracle needed}                 \\ \cmidrule(r){2-6}
                                            	    & Objective $f$                                                          & Objective $g$            & Feasible set $\cZ$                          &	Upper level				&   Lower level                    &                                                                                               \\ \midrule
            MNG \citep{beck2014first}                                        & SC, differentiable                                  & C, smooth                      & Closed              &  Asymptotic              & $\mathcal{O}(1/\epsilon^2)$            &    projection                                                                   \\ \midrule
            BiG-SAM \citep{sabach2017first}                                        & SC, smooth                                   & C, composite        & Closed              &  Asymptotic                                     & $\mathcal{O}(1/\epsilon)$                    &           projection                                            \\ \midrule
           Tseng's method \citep{malitsky2017chambolle}                                         & C, composite                                       & C, composite         & Closed              &  Asymptotic                            & $o(1/\epsilon)$             &      projection                               \\ \midrule
           a-IRG \citep{Kaushik2021}  & C, Lipschitz & C, Lipschitz & Closed & \multicolumn{2}{c}{ ${\mathcal{O}}(\max\{1/\epsilon_f^4,1/\epsilon_g^{4}\})$} & projection \\ \midrule
           \textbf{This paper, Theorem 1}                                         & C, smooth                                                   & C, smooth                      & Compact              & \multicolumn{2}{c}{${\mathcal{O}}(\max\{1/\epsilon_f,1/\epsilon_g\})$}            &   linear solver                                                         \\ \midrule
		       \textbf{This paper, Theorem 2}                                         & Non-C, smooth                                                   & C, smooth                      & Compact              & \multicolumn{2}{c}{${\mathcal{O}}(\max\{1/\epsilon_f^2, 1/(\epsilon_f \epsilon_g) \})$}           &   linear solver                                                          \\ \bottomrule
        \end{tabular}%
    }
\end{table*}

A possible scheme in this case is reformulating problem~\eqref{eq:bi-simp} as a constrained optimization problem with functional constraints and applying primal-dual methods. Specifically, problem~\eqref{eq:bi-simp} can be written as 
\begin{equation}\label{eq:constrained_formulation}
  \min_{\bx\in \reals^d}~f(\bx)\quad \hbox{s.t.}\quad  \bx\in \mathcal{Z},\,g(\bx) \leq g^*,
\end{equation}
where $g^*$ is the optimal value of the lower-level problem. However, a critical issue is that problem~\eqref{eq:constrained_formulation} does not satisfy strict feasibility and hence the Slater's condition fails, 
which is required for most primal-dual methods. %
Even relaxing the constraint ($g(\bx) \leq g^*+\epsilon$) to ensure strict feasibility  would inevitably lead to numerical issues.
In fact, as $\epsilon$ approaches zero and the problem becomes nearly degenerate, the dual optimal variable may tend to infinity, which slows down the convergence and leads to numerical instability \citep{bonnans2013perturbation}
(A detailed discussion about this instability issue is provided in Appendix~\ref{appen:primal-dual}). Therefore, problem~\eqref{eq:bi-simp} cannot be simply treated as a classic constrained optimization problem and calls for new theories and algorithms tailored to its hierarchical structure.

 Another approach to solving problem~\eqref{eq:bi-simp} is the Tikhonov-type regularization~\citep{Tikhonov1977}, where 
the objective functions of both levels are combined using a regularization parameter $\sigma\!>\!0$
to form a single-level problem. It is known that as $\sigma\! \to\! 0$, any cluster point of the solutions of the regularized single-level problem is a solution to the bilevel problem in~\eqref{eq:bi-simp}. In fact, under certain assumptions as shown in \citep{Friedlander2008,Dempe2021}, the solution set of problem \eqref{eq:bi-simp} exactly matches with the regularized problem for sufficiently small $\sigma$. Alas, checking such conditions and finding the threshold are often difficult in practice. To avoid this issue, \cite{Cabot2005} and \cite{solodov2007explicit} proposed  adjusting the regularization parameter $\sigma$ dynamically and proved an asymptotic convergence guarantee.
Along another line of research, several works~\citep{yamada2001,xu2004viscosity} have studied the more general problem of solving a variational inequality over the fixed-point set of a nonexpansive mapping, to which the problem in~\eqref{eq:bi-simp} is a special case. In particular, the hybrid steepest descent method by~\cite{yamada2001} and the sequential averaging method (SAM) by~\cite{xu2004viscosity} converge asymptotically to the optimal solution when the parameters are properly chosen. However, these results fail to provide any non-asymptotic guarantee for either the upper- or lower-level objectives.

More recently, there has been a surge of interest in establishing non-asymptotic convergence rates for problem~\eqref{eq:bi-simp}. One of the first methods of this kind is the minimal norm gradient (MNG) method proposed by \cite{beck2014first}. When the upper-level function $f$ is strongly-convex and the lower-level function $g$ is convex and smooth, they showed that MNG converges asymptotically to the optimal solution and achieves a complexity bound of $\bigO(1/\epsilon^2)$ in terms of the lower-level objective value. Subsequently, built on the SAM framework,  Bilevel Gradient SAM (BiG-SAM) was proposed by \cite{sabach2017first}  and it was shown to achieve a complexity of $\bigO(1/\epsilon)$ for the lower-level problem; see also \cite{shehu2021inertial} for a related method. \cite{malitsky2017chambolle} studied a version of Tseng's accelerated gradient method that obtains a convergence rate of $o(1/k)$ for the lower-level problem.
However, these prior works only establish convergence rates for the lower-level problem, while the rate for the upper-level objective is missing. 
The only exception is the work by \cite{Kaushik2021}: when $f$ and $g$ are convex and Lipschitz continuous, they showed that an iterative regularization-based method achieves a convergence rate of $\bigO(1/k^{0.5-b})$ for the upper-level objective and a rate of $\bigO(1/k^{b})$ for the lower-level, where $b\in(0,0.5)$ is a user-defined parameter. As stated in Table~\ref{tab:bilevel}, if one sets $b=0.25$ to balance the two rates, then finding a solution that is $\epsilon_f$-optimal  for the upper-level problem and $\epsilon_g$-optimal for the lower-level problem would require a complexity of ${\mathcal{O}}(\max\{1/\epsilon_f^4,1/\epsilon_g^{4}\})$.

\textbf{Contributions.} In this paper, we present a novel conditional gradient-based bilevel optimization (CG-BiO) method with tight non-asymptotic guarantees for both upper- and lower-level problems. At each iteration, our proposed CG-BiO method uses a cutting plane to locally approximate the solution set of the lower-level problem, and then combines it with a CG-type update on the upper-level objective. %
Our theoretical guarantees for CG-BiO are the following: 
\vspace{-1mm}
\begin{itemize}
\vspace{-1mm}
  \item When the upper-level function $f$ is convex, we show that CG-BiO finds $\hat{\bx}$ that satisfies $f(\hat{\bx})-f^* \leq \epsilon_f$ and $g(\hat{\bx})-g^* \leq \epsilon_g$
  within ${\mathcal{O}}(\max\{1/\epsilon_f,1/\epsilon_g\})$ iterations, where $f^*$ is the optimal value of problem~\eqref{eq:bi-simp} and $g^*$ is the optimal value of the lower-level problem. This guarantee matches the best-known results in terms of the lower-level objective and is optimal for bilevel projection-free methods.
  \vspace{-1mm}
  \item When $f$ is non-convex, CG-BiO finds $\hat{\bx}$ that satisfies $\FW(\hat{\bx}) \leq \epsilon_f$ and $g(\hat{\bx})-g^* \leq \epsilon_g$
  within  ${\mathcal{O}}(\max\{1/\epsilon_f^2,1/(\epsilon_f\epsilon_g)\})$ iterations, where $\FW(\hat{\bx})$ is the Frank-Wolfe (FW) gap function (cf. \eqref{eq:FW_gap}). 
  \vspace{-1mm}
  \item With an additional $r$-th-order ($r\geq 1$) H\"olderian error bound assumption on the lower-level problem, CG-BiO finds $\hat{\bx}$ with $|f(\hat{\bx})-f^*|\leq \epsilon_f$ within $\bigO(1/\epsilon_f^r)$ iterations in the convex case, and $\hat{\bx}$ with $|\FW(\hat{\bx})|\leq \epsilon_f$ within $\bigO(1/\epsilon_f^{r+1})$ iterations in the non-convex case.
\end{itemize}

\vspace{-2mm}
It is worth noting that the state-of-the-art methods for solving simple bilevel problems (stated in Table~\ref{tab:bilevel}) require projection onto the set $\mathcal{Z}$ at each iteration. In contrast, as our proposed method is a CG-based method, it requires access to a linear solver instead of projection %
at each iteration, which is suitable for the settings where projection is computationally costly; e.g., when $\mathcal{Z}$ is a polyhedron.

 \textbf{Additional Related Work.} In the general form of bilevel problems, the upper-level function $f$ may also depend on an additional variable $\bw\in \reals^m$ that in turn influences the lower-level problem:
\begin{equation}\label{eq:optimstic_bilevel}
  \min_{\bx\in \reals^d,\bw\in \reals^m}~f(\bx,\bw)\quad \hbox{s.t.}\quad  \bx\in\argmin_{\bz\in \cZ}~g(\bz,\bw).
\end{equation}
Problem~\eqref{eq:optimstic_bilevel} has been studied in depth, and we refer the readers to the extensive survey by \cite{dempe2020bilevel}. Note that for any fixed $\bw$, the above problem boils down to a simple bilevel problem in~\eqref{eq:bi-simp}.
In recent years, 
gradient-based methods for problem~\eqref{eq:optimstic_bilevel} have become increasingly popular  
including implicit differentiation~\citep{domke2012generic,pedregosa2016hyperparameter,gould2016differentiating,ji2021bilevel} and iterative differentiation~\citep{maclaurin2015gradient,franceschi2018bilevel}.
However, most of the existing methods work under the assumption that the lower-level problem is strongly convex in $\bz$ for any $\bw$ and thus has a unique minimum. Note that such an assumption would render the simple bilevel problem in \eqref{eq:bi-simp} trivial, as it amounts to solving the lower-level problem only.
More relevant to our work, some concurrent papers consider the case where the lower-level problem
can have multiple minima \citep{liu2020generic,li2020improved,liu2021value,liu2021towards,sow2022constrained,gao2022value}.
They either reformulate problem~\eqref{eq:optimstic_bilevel} as a constrained optimization problem in the same spirit as \eqref{eq:constrained_formulation}, or build upon existing methods (in particular, BiG-SAM by \citet{sabach2017first}) for solving the simple bilevel problem.
Moreover, as they consider a more general problem than ours, their theoretical results are necessarily weaker, providing only asymptotic convergence guarantees or slower rates. 
In this paper, 
we explore a fundamentally different approach for solving the bilevel problem in \eqref{eq:bi-simp} directly and provide tight non-asymptotic convergence guarantees for our method.

\section{PRELIMINARIES}\label{sec:pre}

In this section, we first discuss 
a few motivating examples for problem~\eqref{eq:bi-simp}, which can be generalized to two broader classes of problems: lexicographic optimization~\citep{gong2021automatic} %
and lifelong learning~\citep{chaudhry2018efficient}. {Additional discussions and examples} are provided in Appendix~\ref{sec:additional_example}. Then, we state the required assumptions and notions of optimality that we use for our theoretical results.

\vspace{-1mm}
\subsection{Motivating Examples}\label{subsec:examples}
\vspace{-1mm}
Several machine learning applications consist of a main objective $g$, such as the training loss, and a secondary objective $f$, such as a regularization term or an auxiliary loss. 
In this case, a natural approach is to fully optimize the main objective and use the secondary objective as a criterion to select one of the optimal solutions. 
This approach is also known as lexicographic optimization~\citep{gong2021automatic} and can be formulated as the simple bilevel problem in~\eqref{eq:bi-simp}.
In Examples~\ref{ex:over_regression} and~\ref{ex:fair_class}, we provide instances of such problems. %

In the paradigm of Lifelong Learning, 
the learner faces a stream of possibly related tasks and  
the central theme is to accumulate the knowledge learned from the past and continually improves it given new tasks. 
We can cast this problem as a simple bilevel problem, where the lower-level loss corresponds to samples from seen tasks, while the upper-level loss captures the error on a new task. The goal is to improve the model using the new task while ensuring that it still performs well over the previous tasks. We illustrate an instance of this class of problems in Example 3.

\begin{example}[Over-parameterized regression]\label{ex:over_regression}
In a constrained regression problem, we aim to find a parameter vector $\bbeta\in \reals^{{d}}$ that minimizes the loss $\ell_{\mathrm{tr}}(\bbeta)$ with respect to the training dataset $\cD_{\mathrm{tr}}$.
We also constrain $\bbeta$ to be in some set $\cZ\subseteq\reals^d$ representing some prior knowledge.
For instance, we have 
{$\cZ=\{\bbeta\mid\|\bbeta\|_1\leq \lambda\}$} for some $\lambda>0$ in a sparse regression problem.
Without an explicit regularization, an over-parameterized regression problem over the training dataset possesses multiple global minima. In fact, while any optimization algorithm can achieve one of these many global minima, not all optimal regression coefficients perform equally. Hence, one can consider a secondary objective, such as the loss over a validation set $\cD_{\mathrm{val}}$, to select one from the minimizers of the training loss. This leads to the following bilevel problem:
\begin{equation}\label{eq:over_regression}
  \begin{aligned}
    & \min_{\bbeta\in \reals^d}~f(\bbeta)\triangleq\ell_{\mathrm{val}}(\bbeta)\\
    &\;\,\hbox{s.t.}\quad  \bbeta\in\argmin_{\bz\in \cZ}~g(\bz)\triangleq\ell_{\mathrm{tr}}(\bz).
   \end{aligned}
\end{equation}
We note that problem~\eqref{eq:over_regression} can also appear as a subproblem in hyperparameter selection problems~\citep{gao2022value}.
In this case, both the upper-level and lower-level objectives are smooth and convex if the loss $\ell$ is smooth and convex. 
\end{example}

\begin{example}[Fair classification]\label{ex:fair_class}
  {
    In a binary classification problem, we aim to find a mapping from the feature vectors $\bx_i$ to the target labels $y_i$.
    Due to the bias in the dataset, standard training procedures could lead to a model that discriminates against certain social groups. To alleviate this issue, we can use a fairness metric as a secondary objective to promote fairness in the decision of the model.  
    One common criterion is the $p$\%-rule: given a sensitive attribute $v$ such as race or sex, we require that for any $a$ and $b$,
    \begin{equation*}
        \min \left( \frac{\Pr(\hat{y}=1 \,\vert\,\bx,v=a)}{\Pr(\hat{y}=1 \,\vert\,\bx,v=b)}, \frac{\Pr(\hat{y}=1 \,\vert\,\bx,v=b)}{\Pr(\hat{y}=1 \,\vert\,\bx,v=a)}\right) \geq \frac{p}{100},
    \end{equation*}
    where $\hat{y}$ is the prediction of the model. 
    However, this objective is hard to optimize and hence we use the covariance as a surrogate loss as suggested in~\citep{zafar2017fairness,gong2021automatic}.
    Let $h(\bx;\bbeta)$ be the output of the model parameterized by $\bbeta$, and consider the following problem:
    \begin{equation*}
      \min_{\bbeta\in \reals^d}~(\mathrm{cov}(h(\bx;\bbeta),v))^2\quad \hbox{s.t.}\quad  \bbeta\in\argmin_{\bz\in \cZ}~\ell_{\mathrm{tr}}(\bz).
  \end{equation*}
  Specifically, we aim to minimize the correlation between our prediction model and the sensitive feature $v$ without sacrificing its performance over the training set. 
   }
\end{example}

\begin{example}[Dictionary learning]\label{ex:dict}
    The goal of dictionary learning is to learn a concise representation of the input data from a massive dataset. 
    Let $\bA\!=\!\{\ba_1,\dots,\ba_n\}$ denote a dataset of $n$ points with $\ba_i\!\in\! \reals^m$ for any $i\!\in\!\cN\triangleq\! \{1,\hdots,n\}$. We aim to find a dictionary $\bD=[\bd_1,\dots,\bd_p]\in \reals^{m\times p}$ such that each data point $\ba_i$ can be well approximated by a linear combination of a few basis vectors in $\bD$.  
  A common approach is to formulate this as the following non-convex optimization problem \citep{kreutzdelgado2003dictionary,yaghoobi2009dictionary,rakotomamonjy2013direct,bao2016dictionary}: 
    \begin{equation}\label{eq:dict_learning_low}
    \begin{aligned}
      &\min_{\bD\in \reals^{m\times p}}\min_{\bX\in \reals^{p\times n}}~\frac{1}{2n}\sum_{i\in\cN} \|\ba_i-\bD \bx_i\|_2^2\\ & \hbox{s.t.}\quad \|\bd_j\|_2\leq 1, j=1,\dots,p;\;\|\bx_i\|_1\leq \delta,i\in\cN.
    \end{aligned}
    \end{equation}
   Note that we normalize the basis vectors to have bounded $\ell_2$-norm and impose $\ell_1$-norm constraints to encourage sparsity in $\{\bx_i\}_{i=1}^n$. Further, we refer to $\bX = [\bx_1,\dots,\bx_n]\in \reals^{p\times n}$ as the coefficient matrix.

    In real applications, the data points typically arrive sequentially and the underlying representation may be gradually evolving. Thus, it is desirable to update our dictionary in a continuous manner. Suppose that we already have learned a dictionary $\hat{\bD}\in \reals^{m \times p}$ and the corresponding coefficient matrix $\hat{\bX}\in \reals^{p\times n}$ 
    for the dataset $\bA$.   
    When a new dataset $\bA' =\{\ba'_1,\dots,\ba'_{n'}\}$ arrives, 
    we hope to expand our dictionary by learning new basis vectors from $\bA'$ while retaining the learned information in $\hat\bD$. To achieve so, we aim to find the dictionary $\tilde{\bD}\in \reals^{m\times q}$ ($q>p$) and the coefficient matrix $\tilde{\bX}\in \reals^{q\times n'}$ for the new dataset $\bA'$, and at the same time enforce $\tilde{\bD}$ to perform well on the old dataset $\bA$ together with the learned coefficient matrix $\hat{\bX}$.
    This leads to the following bilevel problem:
    \vspace{-0.5mm}
    \begin{equation}\label{eq:dict_learning}
      \begin{aligned}
        &\min_{\tilde{\bD}\in \reals^{m\times q}}\min_{\tilde{\bX}\in \reals^{q\times n'}}  f(\tilde{\bD},\tilde{\bX}) \\ &\text{s.t.} \quad \|\tilde{\bx}_k\|_1\leq \delta,k=1,\dots,n';\;\tilde{\bD}\in  \argmin_{\|\tilde{\bd}_j\|_2\leq 1}~g(\tilde{\bD}),
        \end{aligned}
    \end{equation}
       \vspace{-0.5mm}
   where the objective $f(\tilde{\bD},\tilde{\bX})\triangleq\frac{1}{2n'}\sum_{k=1}^{n'} \| \ba_k'\!-\! \tilde{\bD} \tilde{\bx}_k\|^2_2$ is the average reconstruction error on the new dataset $\bA'$, the lower-level objective $g(\tilde{\bD})\triangleq\frac{1}{2n}\sum_{i=1}^n \|\ba_i-\tilde{\bD} \hat{\bx}_i\|_2^2$ is the error on the old dataset $\bA$, and with a slight abuse of notation we let $\hat{\bx}_i$ denote the extended vector in $\reals^q$ by appending zeros at the end. 
  Note that in problem~\eqref{eq:dict_learning}, the upper-level objective is non-convex while the lower-level objective is convex with multiple minima. 
 \end{example}

 \vspace{-1mm}
\subsection{Assumptions and Definitions}\label{subsec:assumption}
 \vspace{-1mm}
 
We focus on the case where the lower-level function $g$ is smooth and convex, while the upper-level function $f$ is smooth but not necessarily convex. Formally, we make the following assumptions.
\begin{assumption}\label{assum:smooth}
   Let $\|\cdot\|$ be an arbitrary norm on $\reals^d$ and $\|\cdot\|_*$ be its dual norm.  We assume 
  \begin{enumerate}[(i)]
  \vspace{-1mm}
    \item $\cZ\subset\reals^d$ is convex and compact with diameter $D$, i.e., $\|\bx-\by\| \leq D$ for all $\bx,\by \in \cZ$. 
    \item $g$ is convex and continuously differentiable on an open set containing $\cZ$, and its gradient is Lipschitz with constant $L_g$, i.e.,$
      \|\grad g(\bx)-\grad g(\by)\|_* \leq L_g\|\bx-\by\|
  $ for all $\bx,\by \in \cZ$.
    \item $f$ is continuously differentiable and its gradient is Lipschitz with constant $L_f$. 
  \end{enumerate}
  \end{assumption}
  
  \vspace{2mm}
  \begin{remark}
    Instead of the Lipschitz gradient assumptions above, we may  assume that $f$ and $g$ have bounded \emph{curvature constants}. Such an assumption is common in the analysis of the CG method and has the advantage of being affine-invariant, e.g., see \citep{jaggi2013revisiting,lacoste2016convergence}.
  \end{remark}
  
  Throughout the paper, we use $g^* \triangleq \min_{\bz\in \cZ} g(\bz)$ and $\cX^*_g \triangleq \argmin_{\bz\in \cZ} g(\bz)$ to denote the optimal value and the optimal solution set of the lower-level problem, respectively. Note that by Assumption~\ref{assum:smooth}, the set $\cX^*_g$ is nonempty, compact and convex, but in general not a singleton as $g$ could have multiple minima on $\cZ$.
  Moreover, we use $f^*$ to denote the optimal value and $\bx^*$ to denote an optimal solution of problem~\eqref{eq:bi-simp}, which are guaranteed to exist as $f$ is continuous and $\cX^*_g$ is compact.

For generality, we allow different target accuracies $\epsilon_f$ and $\epsilon_g$ for the upper-level and lower-level problems, respectively, and define an $(\epsilon_f, \epsilon_g)$-optimal solution as follows.
\begin{definition}[$(\epsilon_f,\epsilon_g)$-optimal solution]\label{def:optimal}
  When $f$ is convex, a point $\hat{\bx}\in \mathcal{Z}$ is $(\epsilon_f, \epsilon_g)$-optimal for problem~\eqref{eq:bi-simp} if  
  \begin{equation*}
    f(\hat{\bx}) - f^* \leq \epsilon_f\quad \text{and}\quad g(\hat{\bx}) - g^* \leq \epsilon_g.
  \end{equation*}
  When $f$ is non-convex, $\hat{\bx}\in \mathcal{Z}$ is $(\epsilon_f, \epsilon_g)$-optimal if  
  \begin{equation*}
    \FW (\hat{\bx}) \leq \epsilon_f\quad \text{and}\quad g(\hat{\bx}) - g^* \leq \epsilon_g,
  \end{equation*}
  where $\FW(\hat{\bx})$ is the FW gap \citep{jaggi2013revisiting,lacoste2016convergence} defined by
  \begin{equation}\label{eq:FW_gap}
    \FW(\hat{\bx})\triangleq \max_{\bs\in \cX^*_g}\{\fprod{\grad f(\hat{\bx}),\hat{\bx}-\bs}\}.
  \end{equation}
  \end{definition} 

\section{PROPOSED ALGORITHM}

Before stating our proposed method, we start by the standard CG method~\rj{\citep{Frank1956,levitin1966constrained}} for solving problem~\eqref{eq:bi-simp}. Recall that $\cX_g^*$ denotes the solution set of the lower-level problem.  
If we assume $\bx_0\in \cX_g^*$, then the CG update at iteration $k$ is given by 
\vspace{-1mm}
\begin{equation*}
\vspace{-1mm}
  \bx_{k+1} = (1-\gamma_k)\bx_k+\gamma_k \bs_k,
\end{equation*}
where 
\vspace{-1mm}
\begin{equation}\label{eq:standard_CG}
\vspace{-1mm}
    \bs_k = \argmin_{\bs\in \cX_g^*}~ \fprod{\grad f(\bx_k),\bs},
\end{equation}
and $\gamma_k\in[0,1]$ is the stepsize. However, the main challenge here is that the solution set $\cX_g^*$ for the lower-level problem is not explicitly given,
and hence the linear minimization required in \eqref{eq:standard_CG} is computationally intractable. Moreover, the standard CG method needs to be initialized with a feasible point. In this case, $\bx_0$ has to be an optimal solution of the lower-level problem, which is hard to guarantee in general---in finite number of iterations one may not be able to find an \textit{exact} optimal solution for the lower-level problem. Similar issues also hold if we try to use projection-based methods such as projected gradient descent to solve problem~\eqref{eq:constrained_formulation}. %

\begin{algorithm}[t!]
	\caption{CG-based bilevel optimization (CG-BiO)}\label{alg:bilevel}
	\begin{algorithmic}[1]
	\STATE \textbf{Input}: Target accuracy  $\epsilon_f,\epsilon_g>0$, stepsizes $\{\gamma_k\}_k$
	\STATE \textbf{Initialization}: Set $\bx_0\in \cZ$ s.t. $ g(\bx_0)-g^*\leq \epsilon_g/2$
	\FOR{$k = 0,\dots,K-1$}
		\STATE Compute $\bs_k \gets \argmin_{\bs\in \cX_k} \fprod{\grad f(\bx_k),\bs}$\\ $\cX_k\!\triangleq\! \{\bs\in \cZ: \fprod{\grad g(\bx_k), \bs-\bx_k}\!\leq\! g(\bx_0)-g(\bx_k)\}$
		\vspace{1mm}
		\IF{$\fprod{\grad f(\bx_k),\bx_k-\bs_k} \leq \epsilon_f$ and \\ $\fprod{\grad g(\bx_k),\bx_k-\bs_k}\leq \epsilon_g/2$ %
		}
			\STATE Return $\bx_k$ and STOP
		\ELSE
			\STATE $\bx_{k+1}\gets (1-\gamma_k)\bx_k+\gamma_k \bs_k$
		\ENDIF
	\ENDFOR
	\end{algorithmic}
	\end{algorithm}

Our key idea is to perform a CG update over a local approximation set $\cX_k$ at the $k$-th iteration instead of the hard-to-characterize set $\cX^*_g$. %
To this end, we borrow the idea of \emph{cutting plane} from the optimization literature~\citep{Boyd2018} and let $\cX_k$ be the intersection of $\cZ$ and the halfspace $\cH_k$: 
\begin{align}\label{eq:X_k}
  &\cX_k \triangleq \cZ \cap \cH_k,\\ 
  &%
  \cH_k\triangleq \{\bs\in {\reals^d}\mid \fprod{\grad g(\bx_k), \bs-\bx_k}\leq g(\bx_0)-g(\bx_k)\}.\nonumber
\end{align}
Indeed $\cX_k$ is potentially more tractable than $\cX_g^*$, as the difficult nonlinear inequality $g(\bx)\leq g^*$ in \eqref{eq:constrained_formulation} is replaced by a single linear inequality. Also,  
by using the convexity of $g$, we can show that the hyperplane $\mathcal{H}_k$ eliminates those points known to have a larger value than $g(\bx_0)$. 
Thus, if we initialize our algorithm such that $\bx_0$ is near-optimal for the lower-level problem, the linear inequality in \eqref{eq:X_k} ensures improvement in terms of the lower-level function. 
Further, this also implies that $\cX_k$ contains the solution set $\cX_g^*$, so we are guaranteed to make progress on the upper-level objective $f$.   
We formalize this claim in the following lemma. 
\begin{lemma}\label{lem:subset}
Recall $\cX_g^*$ as the solution set   for the lower-level problem in \eqref{eq:bi-simp} and recall the definition of the set $\cX_k$ in \eqref{eq:X_k}. Then,	for any $k\geq 0$, we have $\cX^*_g  \subseteq  \cX_k$. 
\end{lemma}

Now we are ready to introduce our proposed CG-BiO method. We first initialize $\bx_0\in \cZ$ as a near-optimal solution for the lower-level problem, i.e., $g(\bx_0)-g^* \leq \epsilon_g/2$ for some prescribed accuracy $\epsilon_g$. This can be done by running the standard CG method on the lower-level problem, which requires at most $\bigO(1/\epsilon_g)$ iterations. %
Once the initialization step is done, we simply run CG with respect to the approximation sets $\cX_k$. Specificallly, at iteration~$k$, we solve the following subproblem over the set $\cX_k$ defined in~\eqref{eq:X_k}:
\vspace{-1mm}
\begin{equation}\label{eq:subproblem-lp}
\vspace{-1mm}
  \bs_k = \argmin_{\bs\in \cX_k}~ \fprod{\grad f(\bx_k),\bs},%
\end{equation}
{and update the iterate by $\bx_{k+1} = (1-\gamma_k)\bx_k+\gamma_k \bs_k$ with stepsize $\gamma_k \in [0,1]$.}
{We assume access to a linear optimization oracle that returns the solution of the subproblem in~\eqref{eq:subproblem-lp}, which is standard for projection-free methods~\citep{jaggi2013revisiting,lacoste2016convergence,mokhtari2018escape}. In particular,}
if $\cZ$ can be described by a system of linear inequalities, then problem~\eqref{eq:subproblem-lp} corresponds to a linear program and can be solved efficiently by a standard solver {as we will show in our experiments}.
We repeat the process above until we reach an accuracy of $\epsilon_f$ for the upper-level objective and an accuracy of $\epsilon_g$ for the lower-level objective.
The steps of our proposed CG-BiO method are summarized in Algorithm~\ref{alg:bilevel}. 

\section{CONVERGENCE ANALYSIS}\label{sec:analysis}
In this section, we analyze the iteration complexity of our CG-BiO method. 
We first consider the case where the upper-level function $f$ is convex. In this case, we choose the stepsize as $\gamma_k = 2/(k+2)$, %
which is a typical choice in the standard CG method~\citep{jaggi2013revisiting}.
\begin{theorem}[Convex upper-level]\label{thm:convex-upper-bound}
Suppose that Assumption~\ref{assum:smooth} holds and $f$ is convex. 
Let $\{\bx_k\}_{k=0}^{K}$ be the sequence generated by Algorithm \ref{alg:bilevel} with stepsize $\gamma_k  = 2/(k+2)$ for $k\geq 0$. Then we have
\begin{align*}
    f(\bx_K)-f^*\!\leq \frac{2L_f D^2}{K+1},
    \ \
	g(\bx_K)-g^*\leq \frac{2L_g D^2}{K+1}+\frac{\epsilon_g}{2}.
\end{align*}
\end{theorem}

Theorem~\ref{thm:convex-upper-bound} shows that the gap of the upper-level objective can be upper bounded by $\bigO(1/K)$, similar to the convergence bound of standard CG. 
At the same time, the gap of the lower-level objective can also be controlled by a term of order $\bigO(1/K)$ in addition to the initial error $\epsilon_g/2$.  
As a corollary, Algorithm~\ref{alg:bilevel} will return an $(\epsilon_f,\epsilon_g)$-optimal solution when the number of iterations $K$ exceeds
\begin{equation*}
  \max\left\{\frac{2L_f D^2}{\epsilon_f},\frac{4L_g D^2}{\epsilon_g}\right\} = \mathcal{O}\left(\max\left\{\frac{1}{\epsilon_f},\frac{1}{\epsilon_g}\right\}\right).
  \vspace{-0.95mm}
\end{equation*}

Our complexity bound improves over the result by \cite{Kaushik2021}, who considered a different setup where both the upper-level and lower-level functions are Lipschitz but not necessarily smooth.    
Also, comparing with existing works in the same setup, our convergence rate for the lower-level objective matches those by \cite{sabach2017first,malitsky2017chambolle}, while we also provide a non-asymptotic convergence bound for the upper-level objective. To the best of our knowledge, our result provides the best-known bound for the considered setting. We also remark that our rate is tight at least within the family of projection-free methods, since it is known that their worst-case complexity is $\Theta(1/\epsilon_f)$ even for a single-level problem \citep{jaggi2013revisiting,lan2013complexity}. 

{
\vspace{2mm}
\begin{remark}
The initialization step requires $\bigO(1/\epsilon_g)$ iterations, and hence, this additional term does not change the overall complexity.  The same applies for the non-convex case below.
\end{remark}
}

Now we turn to the case where $f$ is non-convex. In this case, we choose the stepsize as a constant depending on the target accuracies as well as the problem parameters.
\begin{theorem}[Non-convex upper-level]\label{thm:nonconvex-upper-bound}
Suppose that Assumption~\ref{assum:smooth} holds. 
Let $\{\bx_k\}_{k=0}^{K-1}$ be the sequence generated by Algorithm \ref{alg:bilevel} with stepsize $\gamma_k = \min\big\{\frac{\epsilon_f}{ L_f D^2},\frac{\epsilon_g}{L_g D^2}\big\}$ for all $k\geq 0$. 
Define $\underf = \min_{\bx \in Z} f(\bx)$. 
Then for $$K \geq \max\Bigl\{\frac{2L_fD^2 (f(\bx_0)-\underf)}{\epsilon_f^2},\frac{2L_g D^2 (f(\bx_0)-\underf)}{\epsilon_f\epsilon_g }\Bigr\},$$ there exists $k^*\in\{0,1,\dots,K-1\}$ such that 
$
    \FW (\bx_{k^*})\leq \epsilon_f
$ and $g(\bx_{k^*})-g^*\leq \epsilon_g$.
\end{theorem}
As a corollary of Theorem~\ref{thm:nonconvex-upper-bound}, the number of iterations required to find an $(\epsilon_f,\epsilon_g)$-optimal solution can be upper bounded by ${\mathcal{O}}(\max\{1/\epsilon^2_f,1/(\epsilon_f\epsilon_g)\})$. We note that the dependence on the upper-level accuracy $\epsilon_f$ also matches that in the standard CG method for a single-level problem \citep{lacoste2016convergence,mokhtari2018escape}.

\blue{We end this section with the following remark. Since the algorithm's output $\hat\bx$ may lie outside of the feasible set $\cX_g^*$, both $f(\hat\bx)-f^*$ in Theorem~\ref{thm:convex-upper-bound} and $\FW(\hat\bx)$ in Theorem~\ref{thm:nonconvex-upper-bound} are not necessarily positive. While this might seem unconventional, we note that %
\cite{Kaushik2021}
also used $f(\hat{\bx})-f^*$ as the performance metric. In fact, this is also common in the literature on constrained optimization, where the generated iterate could be infeasible and thus $f(\hat{\bx})-f^*$ could be negative (see, e.g., \cite{beck2017first}). 
On the other hand, we note that it is in general impossible to prove convergence in terms of $|f(\hat{\bx})-f^*|$ due to a negative result by \cite{chen2023bilevel}. Specifically, for any first-order method and a given number of iterations $K$, they showed that there exists an instance of problem~\eqref{eq:bi-simp} where $|f(\bx_k)-f^*| \geq 1$ for all $0\leq k \leq K-1$. Therefore, to 
provide convergence bounds on $|f(\hat\bx)-f^*|$ or $|\FW(\hat\bx)|$, it is necessary to impose additional regularity conditions on problem~\eqref{eq:bi-simp}, which we discuss in the next section. 
}

\vspace{-1mm}
\subsection{Convergence under H\"olderian Error Bound}\label{subsec:error_bound}
\vspace{-1mm}

\blue{In this section, we complement the convergence results in Theorems~\ref{thm:convex-upper-bound} and~\ref{thm:nonconvex-upper-bound} by giving a lower bound on $f(\hat{\bx})-f^*$ and $\FW(\hat{\bx})$.}
Let $\hat{\bx}$ be an $(\epsilon_f,\epsilon_g)$-optimal solution as defined in Definition~\ref{def:optimal}. Intuitively, since $\hat{\bx}$ is $\epsilon_g$-optimal for the lower-level problem, it should be close to the optimal solution set $\cX^*_g$ under some regularity condition on $g$. As such, we can lower bound $f(\hat{\bx})-f^*$ by using the smoothness of $f$. 
Formally, we assume that the lower-level objective satisfies the H\"olderian error bound, which quantifies the growth rate of the objective value $g(\bx)$ as the point ${\bx}$ deviates from the optimal solution set $\cX_g^*$.   
 
\begin{assumption}\label{assum:Holder}
	The function $g$ satisfies the H\"olderian error bound for some $\alpha>0$ and $r\geq 1$, i.e, 
	\vspace{-0.5mm}
\begin{equation}\label{eq:error_bound}
\frac{\alpha}{r} \: \dist(\bx,\cX_g^*)^r\leq  g(\bx) - g^*, \qquad  \forall \bx\in \cZ,
\vspace{-0.5mm}
\end{equation}
where $\dist(\bx,\cX_g^*)\triangleq  \inf_{\bx'\in \cX_g^*}\norm{\bx-\bx'}$. 
\end{assumption}
We note that the error bound condition in~\eqref{eq:error_bound} is well-studied in the optimization literature (see \citep{pang1997error,bolte2017error,Roulet2020} and the references therein) and is known to hold generally when the function $g$ is analytic and the set $\cZ$ is bounded~\citep{lojasiewicz1959sur,luo1994error}. Two important special cases are: 1) $g$ satisfies \eqref{eq:error_bound} with $r=1$, i.e., $\cX_g^*$ is a set of weak sharp minima of $g$~\citep{burke1993weak,burke2005weak}; 2) $g$ satisfies \eqref{eq:error_bound} with $r=2$ known as quadratic growth condition \citep{drusvyatskiy2018error}. %

Under Assumption~\ref{assum:Holder}, we can establish the following lower bounds on $f(\hat{\bx})-f^*$ and $\FW(\hat{\bx})$. Notably, the following result is an intrinsic property of problem~\eqref{eq:bi-simp} and independent of the algorithm we use. 
\begin{proposition}\label{prop:error_bound}
	Assume that $g$ satisfies 
 Assumption~\ref{assum:Holder}, 
 and define $M= \max_{\bx\in \cX_g^*} \|\grad f(\bx)\|_*$. Then for any $\hat{\bx}$ that satisfies $g(\hat{\bx})-g^*\leq \epsilon_g$, it holds that: 
	\vspace{-1mm}
  \begin{enumerate}[(i)]
  \vspace{-2mm}
    \item If $f$ is convex, then 
    $
      f(\hat{\bx})-f^* \geq -M \left(\frac{r\epsilon_g}{\alpha}\right)^{\frac{1}{r}}$.
    \vspace{-1mm}
    \item If $f$ is non-convex and has $L_f$-Lipschitz gradient, then 
    $
		\FW(\hat{\bx}) \geq -M \left(\frac{r\epsilon_g}{\alpha}\right)^{\frac{1}{r}}- L_f\left(\frac{r\epsilon_g}{\alpha}\right)^{\frac{2}{r}}.
$
\vspace{-1mm}
  \end{enumerate}
  
\end{proposition}

By combining Theorems~\ref{thm:convex-upper-bound} and \ref{thm:nonconvex-upper-bound} with Proposition~\ref{prop:error_bound}, we obtain the following stronger convergence guarantees for the output of our proposed method. 
\begin{corollary}\label{coro:stronger_complexity}
  Suppose that Assumption~\ref{assum:smooth} holds and $g$ satisfies the H\"olderian error bound in Assumption~\ref{assum:Holder} with $\alpha>0$ and $r \geq 1$. Let $M= \max_{\bx\in \cX_g^*} \|\grad f(\bx)\|_*$.
  \vspace{-1mm}
  \begin{enumerate}[(i)]
  \vspace{-2mm}
    \item If $f$ in problem~\eqref{eq:bi-simp} is convex and we set $\epsilon_g = \frac{\alpha}{r}  \left(\frac{\epsilon_f}{M}\right)^r$, 
    then after $K=\bigO({1}/{\epsilon_f^r})$ iterations
    we have 
    $
      |f(\bx_K)-f^*| \leq \epsilon_f$ {and} $g({\bx}_K) - g^* \leq \epsilon_g
    $.
    \vspace{-1mm}
    \item If $f$ in problem~\eqref{eq:bi-simp} is non-convex and we set $\epsilon_g = \min\{\frac{\alpha}{r}  \left(\frac{\epsilon_f}{2M}\right)^r, \frac{\alpha}{r}\bigl(\frac{ \epsilon_f}{2L_f}\bigr)^{r/2}\}$,
    then after $K=\bigO(1/\epsilon_f^{r+1})$ iterations
    there exists $k^*\in\{0,\dots,K-1\}$ such that    $
      |\FW (\bx_{k^*})| \leq \epsilon_f$ {and}  $g(\bx_{k^*})-g^*\leq \epsilon_g
    $.
  \end{enumerate} 
\end{corollary}

\begin{figure*}[t!]
  \centering
  \subfloat[Lower-level gap]{\includegraphics[width=0.25\textwidth]{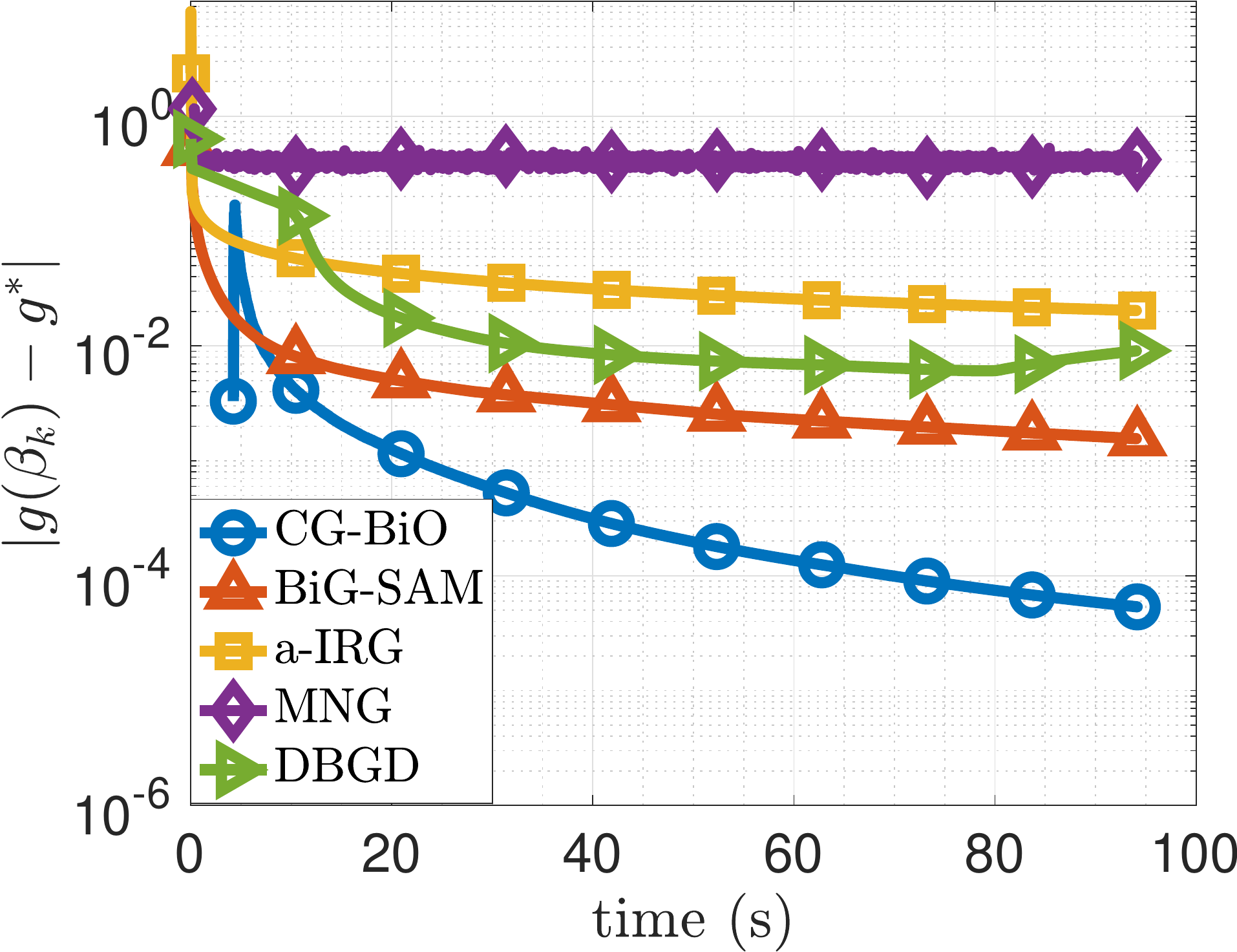}}
  \qquad
  \subfloat[Upper-level gap]{\includegraphics[width=0.25\textwidth]{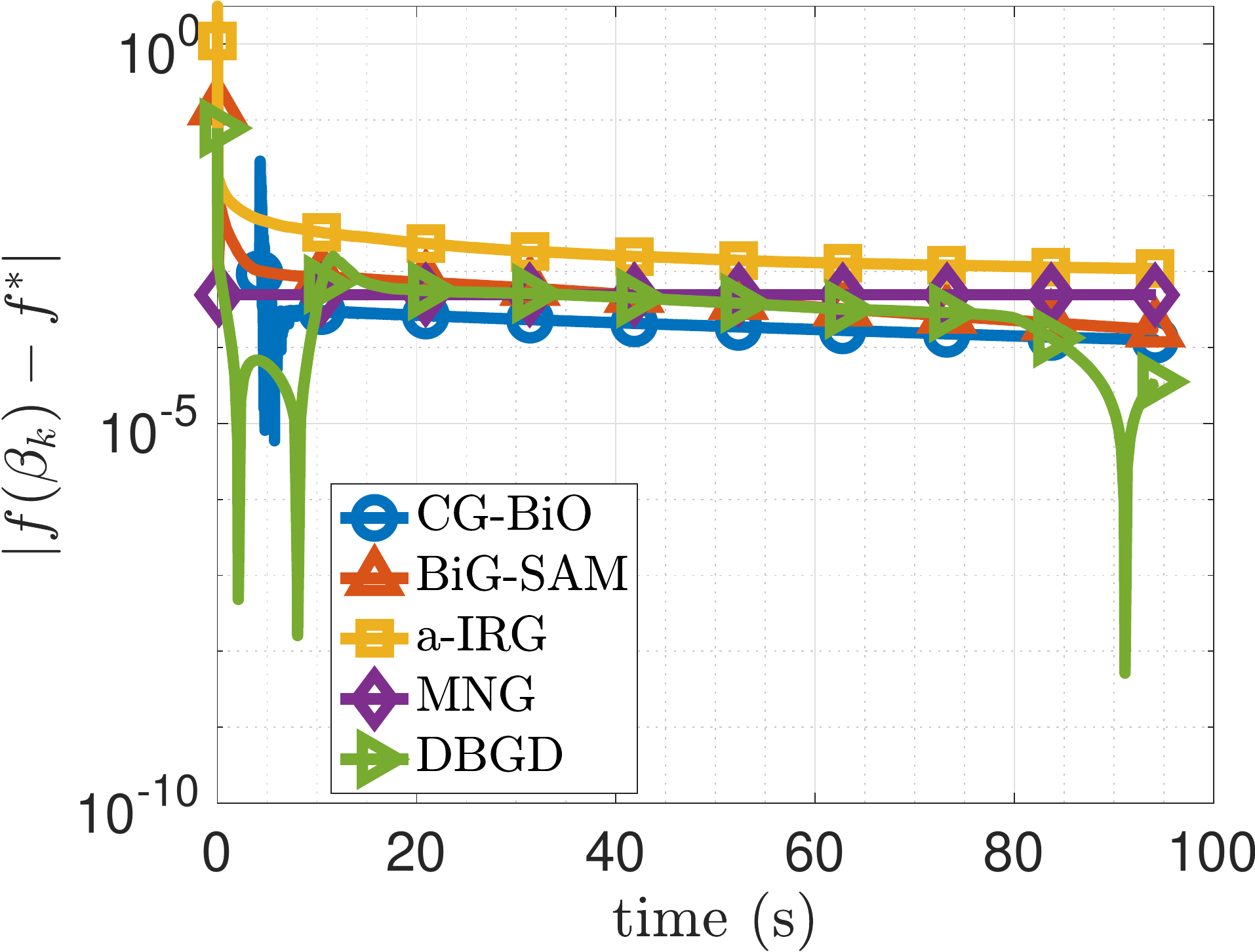}}
  \qquad
  \subfloat[Test error objective]{\includegraphics[width=0.25\textwidth]{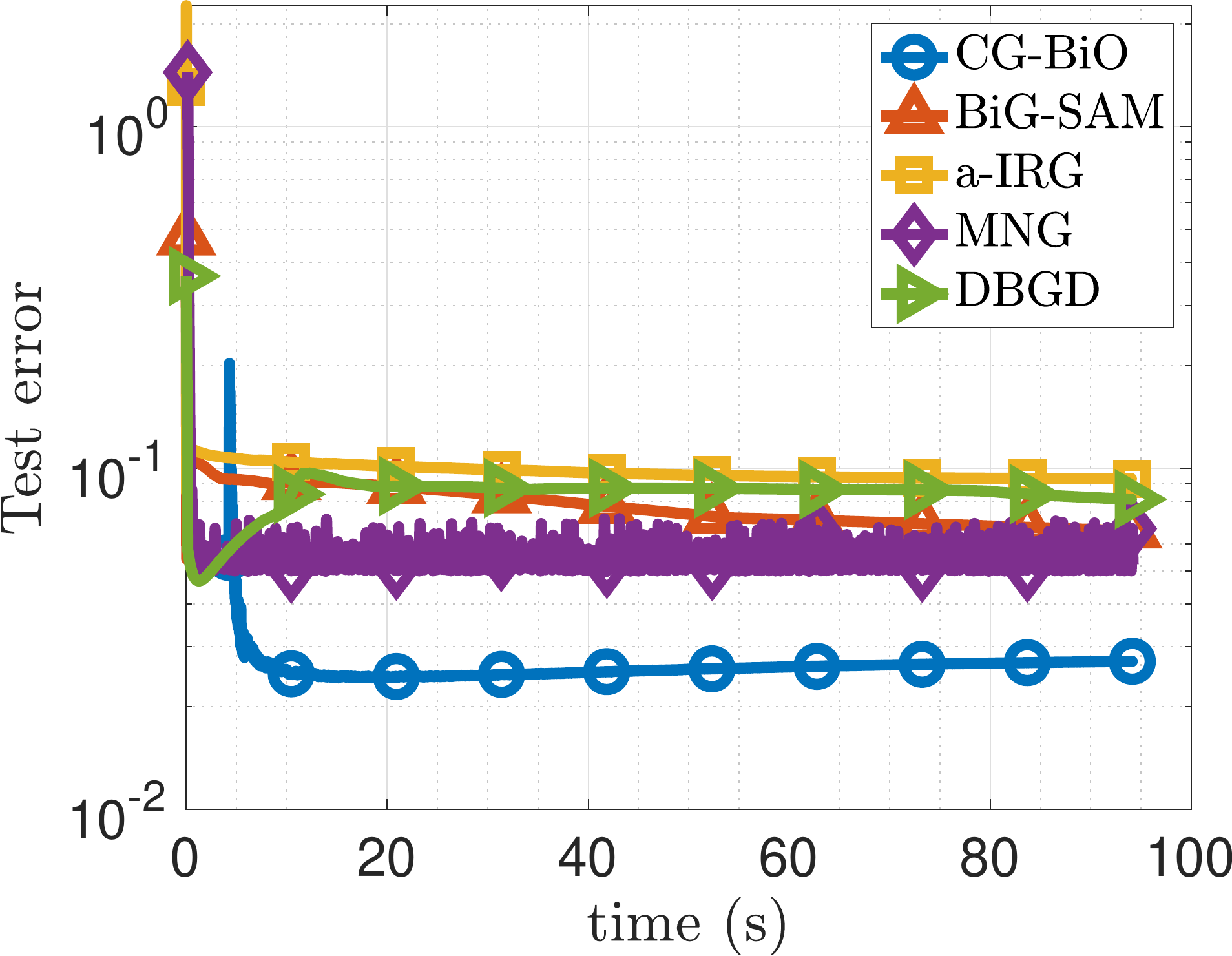}}
   \vspace{-1mm}
  \caption{The performance of CG-BiO {compared with BiG-SAM, a-IRG and MNG} on problem~\eqref{eq:over_regression}. %
  }
  \label{fig:num_appts1}\vspace{-3mm}
\end{figure*}

Corollary~\ref{coro:stronger_complexity} shows that under the $r$-th H\"olderian error bound assumption, we can find an iterate to be $\epsilon_f$-close to optimality within $\bigO(1/\epsilon_f^r)$ iterations in the convex case, and to be $\epsilon_f$-close to stationarity within $\bigO(1/\epsilon_f^{r+1})$ iterations in the non-convex case.

\section{NUMERICAL EXPERIMENTS}\label{sec:numeric}
In this section, we test our method on three different bilevel optimization problems, described in Section~\ref{sec:pre}, with real and synthetic datasets and compare our method 
{with other existing methods in the literature~\citep{beck2014first,sabach2017first,Kaushik2021,gong2021automatic}.}
\subsection{\rj{Over-parameterized Regression}}
For Example~\ref{ex:over_regression}, we consider a sparse linear regression problem on the Wikipedia Math Essential dataset \citep{rozemberczki2021pytorch}, which consists of a data matrix $\bA\in \reals^{{n \times d}}$ {with $n=1068$ instances and $d=730$ attributes}
and an outcome vector $\bb\in \reals^{n}$. 
We assign 60\% of the dataset as the training set $(\bA_{\mathrm{tr}},\bb_{\mathrm{tr}})$, 20\% as the validation set $(\bA_{\mathrm{val}},\bb_{\mathrm{val}})$ and the rest as the test set $(\bA_{\mathrm{test}},\bb_{\mathrm{test}})$. 
Then the lower-level objective in \eqref{eq:over_regression} is the training error $g(\bz)=\frac{1}{2}\|\bA_{\mathrm{tr}}\bz-\bb_{\mathrm{tr}}\|_2^2$, the upper-level objective is the validation error $f(\bbeta)=\frac{1}{2}\|\bA_{\mathrm{val}}\bbeta-\bb_{\mathrm{val}}\|_2^2$, and the constraint set is $\cZ=\{\bbeta\mid \|\bbeta\|_1 \leq \lambda\}$ for some $\lambda>0$ to induce sparsity in $\bbeta$. We also use the test error $\frac{1}{2}\|\bA_{\mathrm{test}}\bbeta-\bb_{\mathrm{test}}\|_2^2$ as our performance metric. 
{Note that the regression problem is over-parameterized since the number of features $d$ is larger than the number of data instances in the training set.}
We compare the performance of our method CG-BiO against the MNG method by \cite{beck2014first}, the Bilevel Gradient SAM (BiG-SAM) by \cite{sabach2017first}, the averaging iteratively regularized gradient (a-IRG) by \cite{Kaushik2021}, \blue{and the Dynamic Barrier Gradient Descent (DBGD) by \cite{gong2021automatic}.} \blue{It is worth noting that while we can implement these methods numerically, some of them are not directly applicable to our problem setting and thus lack any convergence guarantee; see Appendix~\ref{appen:experiment} for more discussions.} %
For benchmarking purposes, we use CVX \citep{cvx,gb08} to solve the lower-level problem and the constrained reformulation in \eqref{eq:constrained_formulation} to obtain the optimal values $g^*$ and $f^*$, respectively. 

In Fig.~\ref{fig:num_appts1}(a), we observe that CG-BiO converges at a faster rate than {the other baseline methods} %
in terms of the lower-level objective, which confirms our theoretical result (cf. Table~\ref{tab:bilevel}). 
Fig.~\ref{fig:num_appts1}(b) and (c) show that 
CG-BiO achieves a smaller upper-level objective gap as well as a smaller test error compared to %
other methods within the same run time.   
Interestingly, after the initial stage, the upper-level objective $f(\bbeta_k)$ of CG-BiO  \emph{increases}, while the optimality gap $|f(\bbeta_k)-f^*|$ \emph{decreases}. 
This suggests that CG-BiO may ``overshoot'' at the beginning due to its relatively large stepsize. Nevertheless, as the number of iterations increases and the level of infeasibility decreases, the upper-level objective of CG-BiO approaches the optimal value of the bilevel problem, which is in line with Proposition~\ref{prop:error_bound}.

\subsection{\rj{Fair Classification}}

We use the Adult income dataset \citep{Dua_2019}  containing 48,842 subjects each with 14 attributes, where the task is to predict whether the annual income of a given subject exceeds \$50K. For efficiency, we randomly sample 2,000 data points as the training set and 1,000 as the test set. We choose ``sex'' as the sensitive attribute $v$. 
Further, we adopt the logistic regression classifier as our model, where the posterior probability is given by $\Pr(\hat{y}_i=1\,\vert\,\bx_i;\bbeta)=1/(1+e^{-\bx_i^\top \bbeta})$ for a given feature vector $\bx_i$ and parameter $\bbeta$. Concretely, the lower-level problem is a sparse logistic regression problem for some $\lambda>0$:
\vspace{-1mm}
\begin{equation}\label{eq:fair_lower_level}
\vspace{-1mm}
  \min_{\bbeta\in \reals^d}~g(\bbeta) \!=\!- \frac{1}{n}\sum_{i=1}^n \log \Pr(\hat{y}_i\! =\! y_i\,\vert\,\bx_i;\bbeta) \ \  \hbox{s.t.}\;\|\bbeta\|_1 \leq \lambda,
\end{equation}
while the upper-level objective is the squared covariance:
\vspace{-1mm}
\begin{equation}\label{eq:fair_upper_level}
\vspace{-1mm}
  f(\bbeta) %
  \!=\! \Big(\frac{1}{n} \sum_{i=1}^n (v_i-\bar{v}) \Pr(\hat{y}_i=1\,\vert\,\bx_i;\bbeta)\Big)^2,
\end{equation}  
where $\bar{v} = \frac{1}{n}\sum_i v_i$. 
Note that the lower-level objective in \eqref{eq:fair_lower_level} is convex while the upper-level in \eqref{eq:fair_upper_level} is non-convex. We numerically verified that the lower-level problem can possess multiple optimal solutions.

\begin{figure*}[t!]
  \centering
  \subfloat[Lower-level gap]{\includegraphics[width=0.23\textwidth]{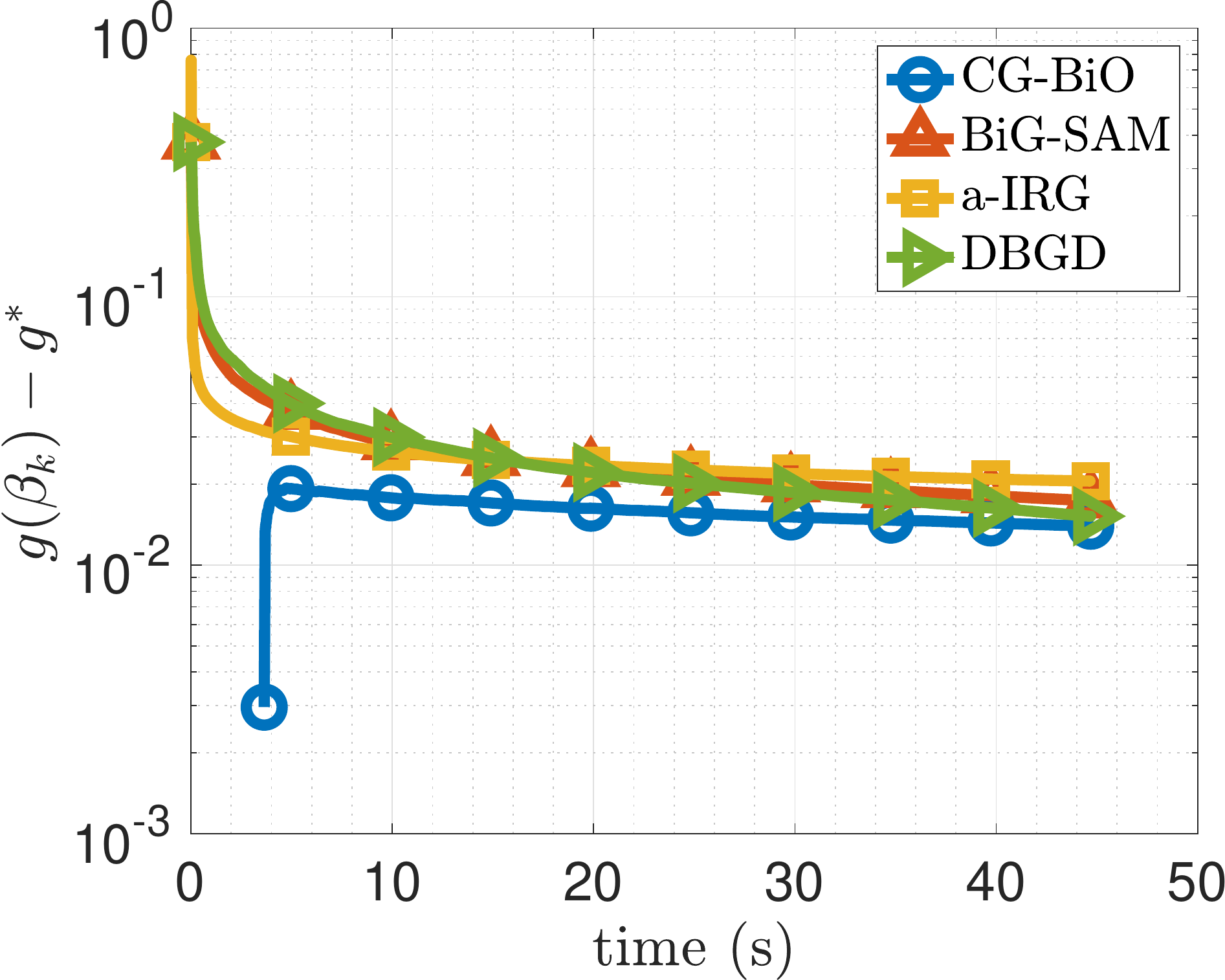}}
  \hfill
  \subfloat[Upper-level objective]{\includegraphics[width=0.23\textwidth]{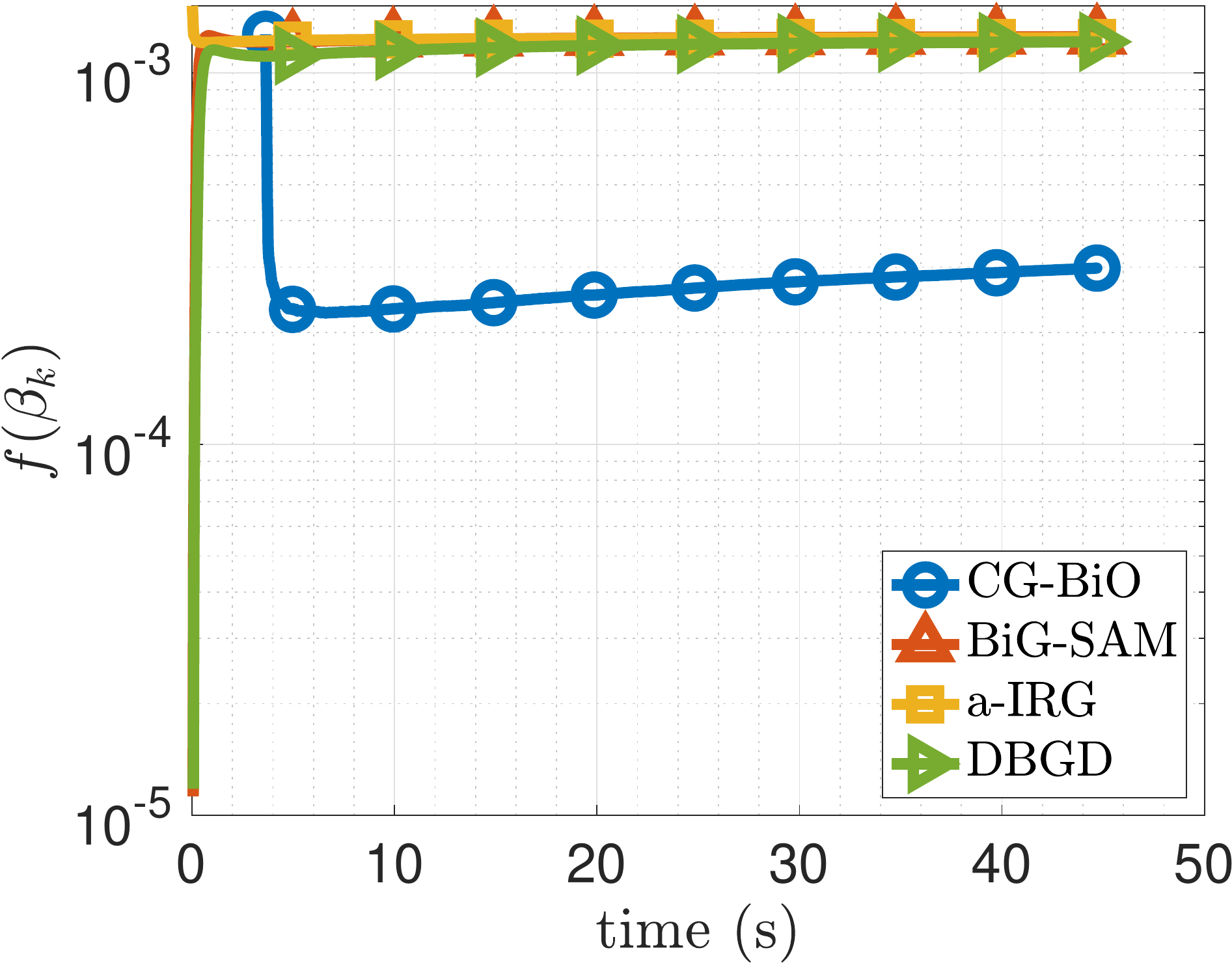}}
  \hfill
  \subfloat[Test set accuracy]{\includegraphics[width=0.23\textwidth]{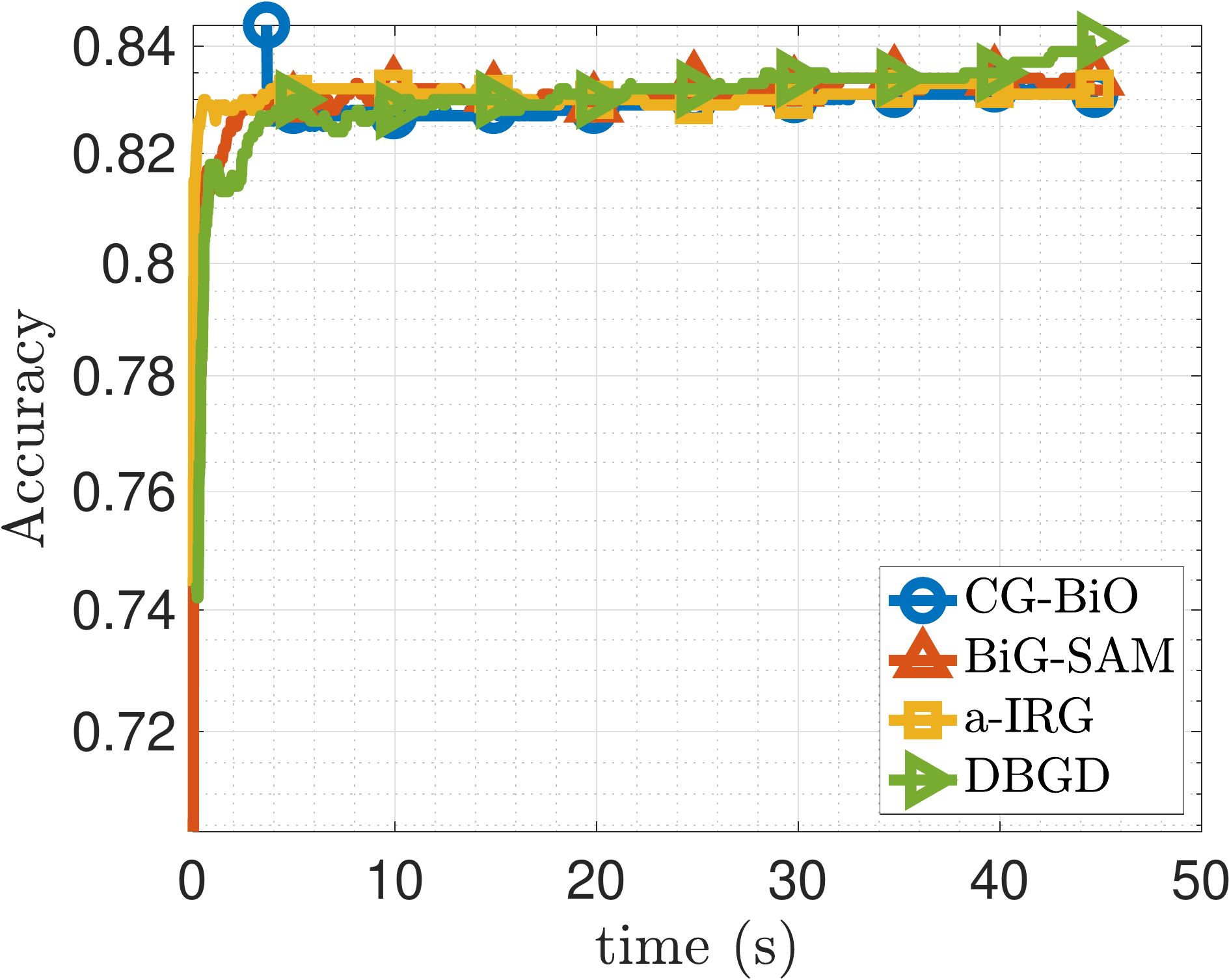}}
  \hfill
  \subfloat[$p$\%-rule]{\includegraphics[width=0.223\textwidth]{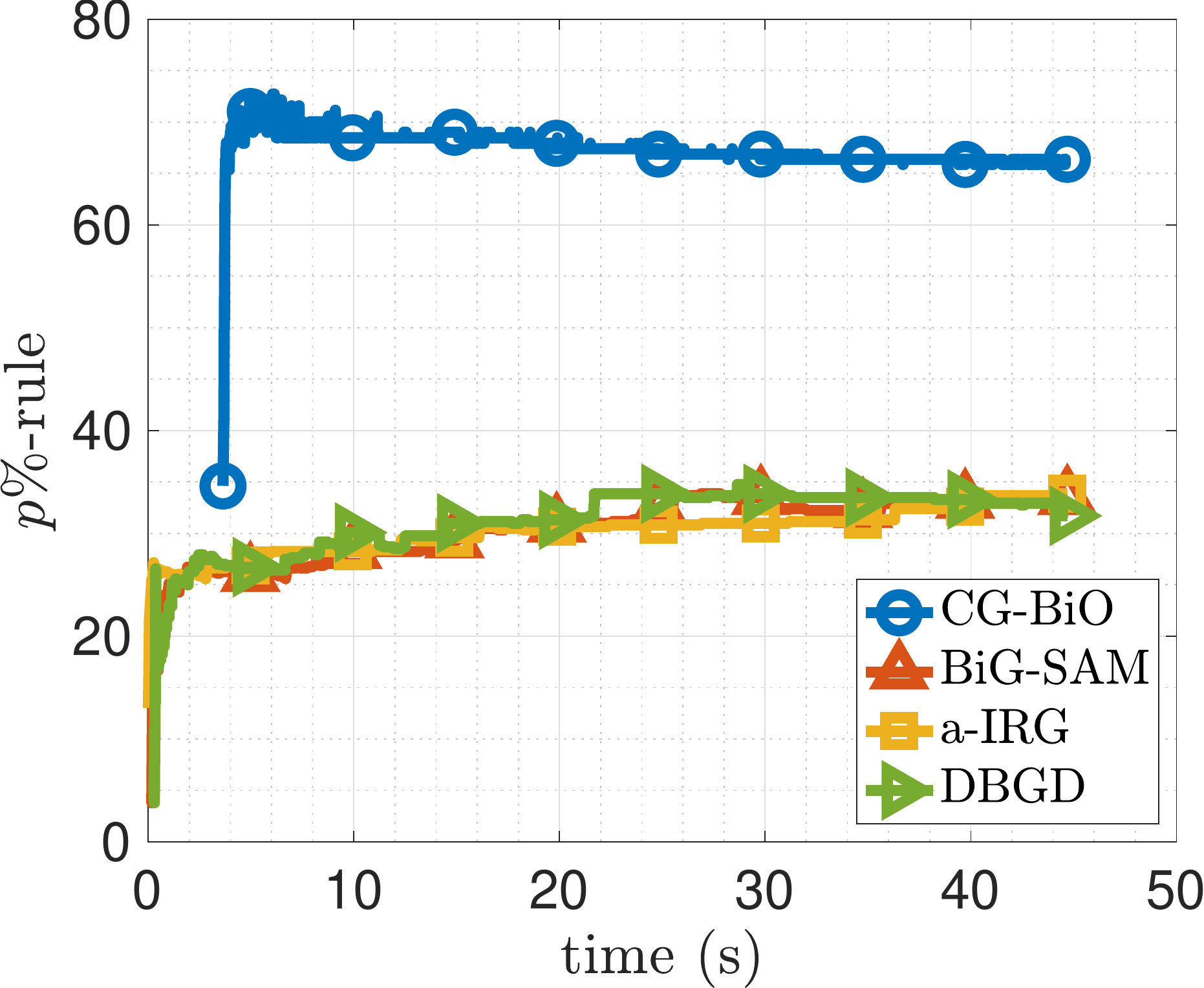}}
  \vspace{-1mm}
  \caption{The performance of CG-BiO {compared with BiG-SAM, a-IRG} on the bilevel problem defined in \eqref{eq:fair_lower_level} and \eqref{eq:fair_upper_level}. %
  }
  \label{fig:fair} %
\end{figure*}

\begin{figure*}[t!]
  \centering
  \subfloat[Lower-level objective]{\includegraphics[width=0.25\textwidth]{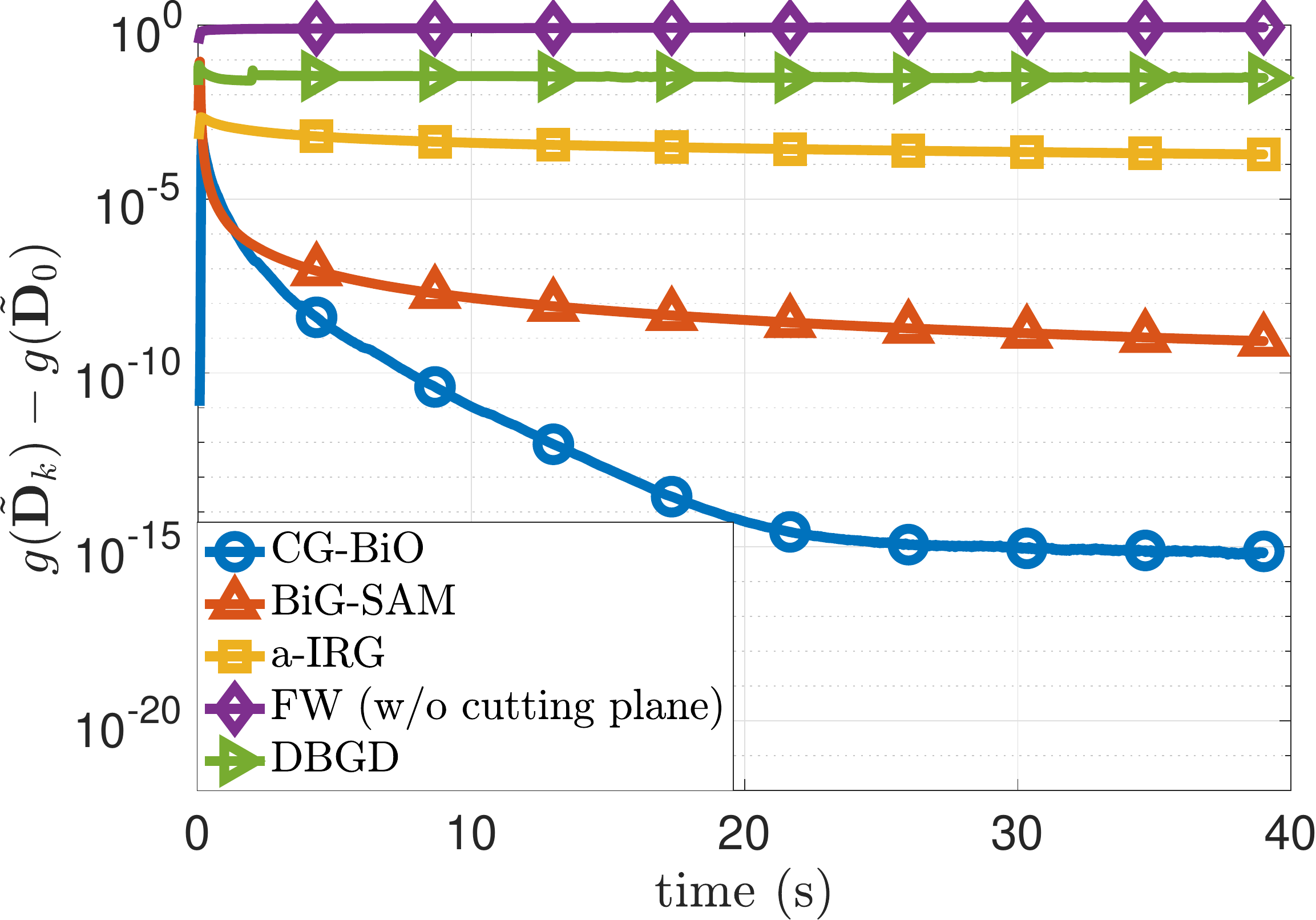}}
  \qquad
  \subfloat[Upper-level objective]{\includegraphics[width=0.25\textwidth]{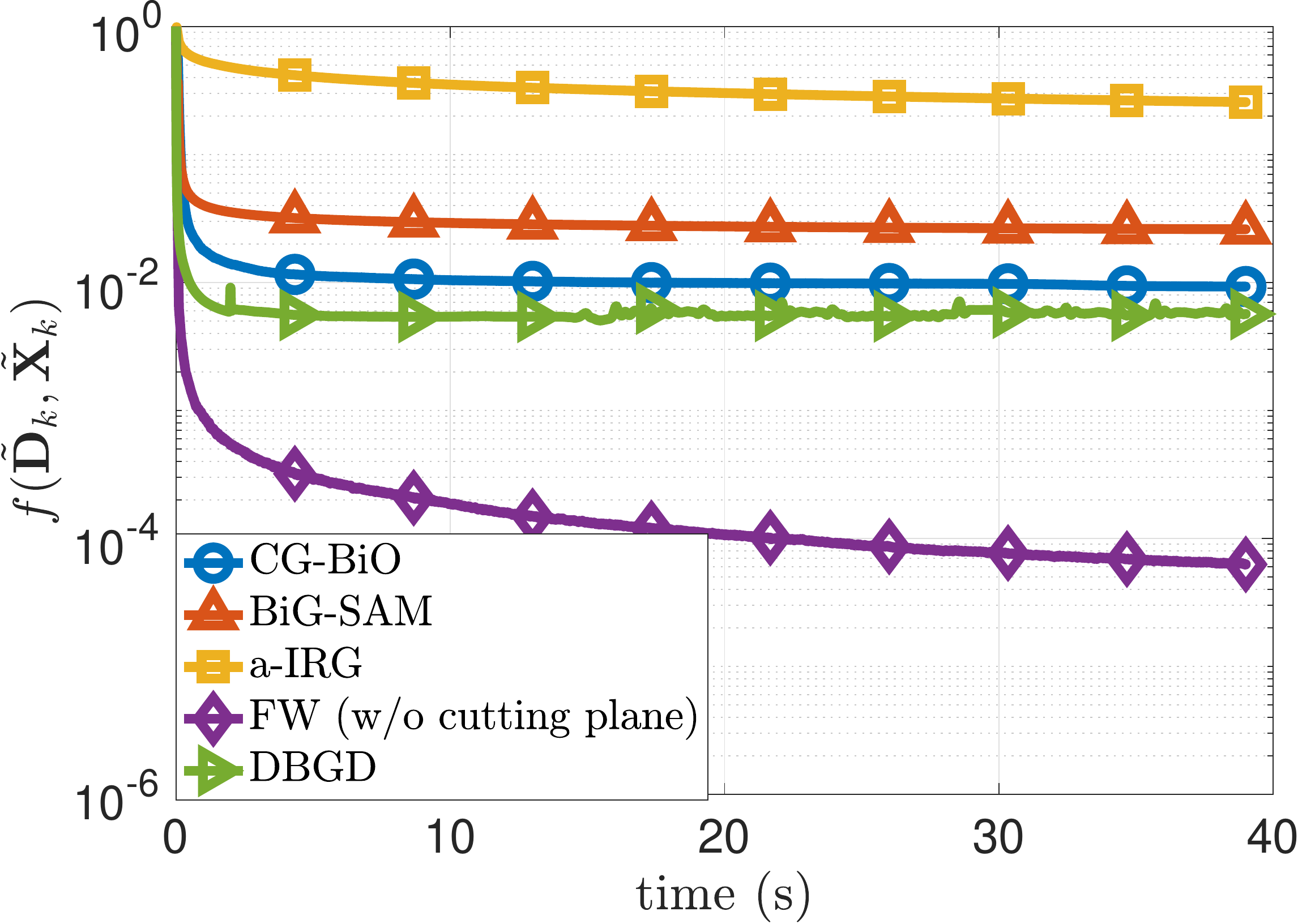}}
  \qquad
  \subfloat[Recovery rate]{\includegraphics[width=0.245\textwidth]{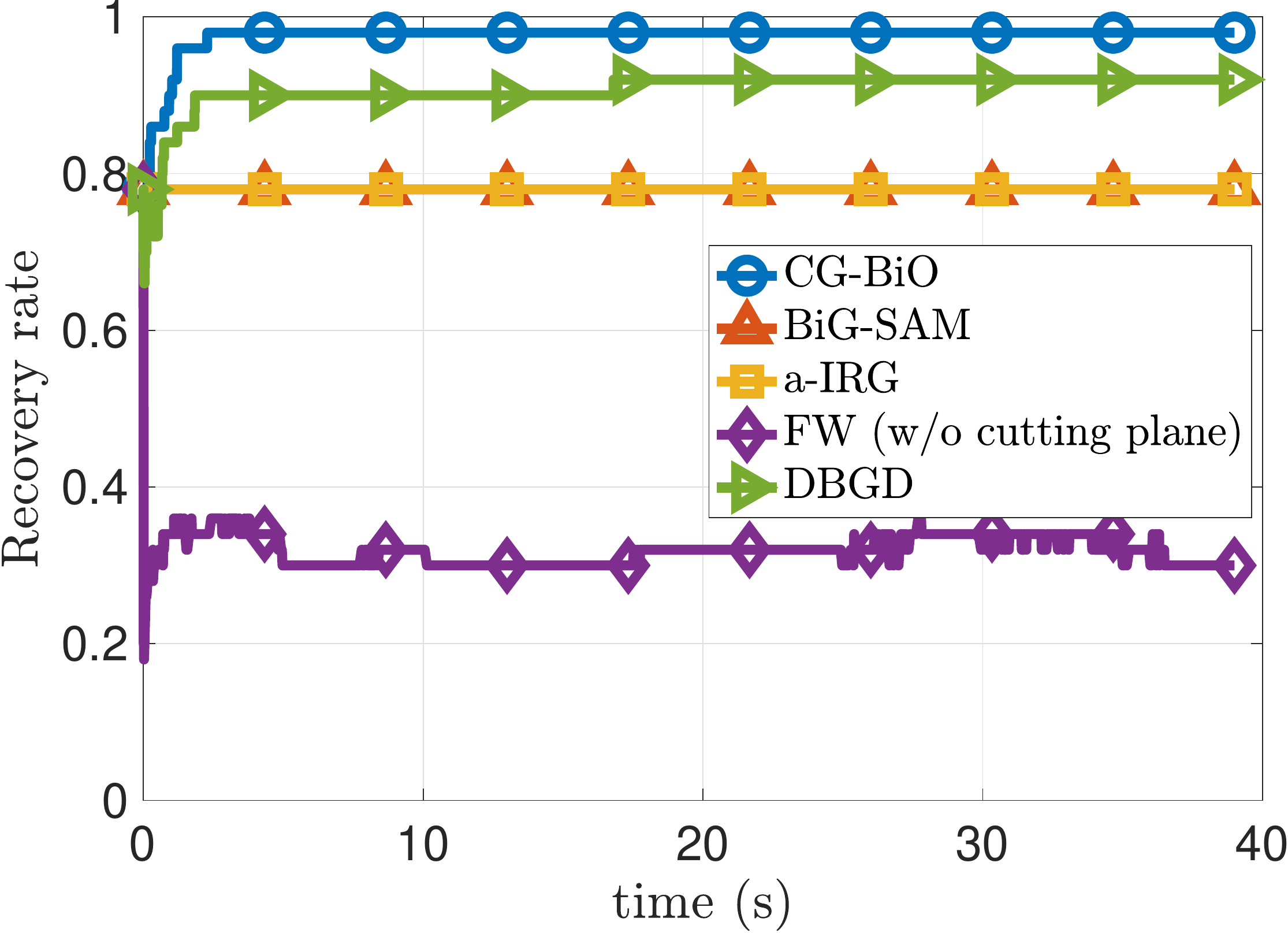}}
  \vspace{-1mm}
  \caption{The performance of CG-BiO {compared with BiG-SAM, a-IRG} and the baseline CG method on problem~\eqref{eq:dict_learning}. %
  }
  \label{fig:nonconvex} %
\end{figure*}

{In Fig.~\ref{fig:fair}, we compare CG-BiO against BiG-SAM, a-IRG \blue{and DBGD}. For this example, we did not implement the MNG method as it is computationally intractable; see Appendix~\ref{appen:experiment}. For a fair comparison, we manually tune the hyperparameters such that all achieve a similar performance on the lower-level problem as shown in Fig.~\ref{fig:fair}(a), and judge their efficiency by the error of the upper-level objective. Fig.~\ref{fig:fair}(b) shows that CG-BiO is able to achieve a smaller upper-level objective (smaller covariance). As a result, we observe in Fig.~\ref{fig:fair}(c) and (d) that it reaches a much better $p$\%-rule with a similar level of accuracy on the test set.  
}

\subsection{\rj{Dictionary Learning}}

We evaluate our CG-BiO method for problem~\eqref{eq:dict_learning} on a synthetic dataset, similar to the setup in \cite{rakotomamonjy2013direct}. 
We first generate the true dictionary $\tilde{\bD}^* \in \reals^{25\times 50}$ consisting of 50 basis vectors in $\reals^{25}$, each of which has its entries drawn from a standard Gaussian distribution and is normalized to have unit $\ell_2$-norm. We further construct the two dictionaries $\bD^*$ and ${\bD'}^*$ consisting of 40 and 20 basis vectors in $\tilde{\bD}^*$, respectively (and hence
they share 10 bases in common). The two datasets $\bA\!=\!\{\ba_1,\dots,\ba_{250}\}$ and $\bA'\!=\!\{\ba'_1,\dots,\ba'_{200}\}$ are generated according to the rules: 
\begin{alignat*}{3}
    \ba_i & = \bD^* \bx_i + \bn_i,&\quad i\!=\!1,\dots,250; \\
    \ba_k' & = {\bD'}^{*} \bx'_k + \bn'_k, &\quad k\!=\!1,\dots,200,
\end{alignat*}
where $\{\bx_i\}_{i=1}^{250},\{\bx_k'\}_{k=1}^{200}$ are sparse coefficient vectors and $\{\bn_i\}_{i=1}^{250},\{\bn'_k\}_{k=1}^{200}$ are random Gaussian noise vectors. Since neither $\bA$ nor $\bA'$ contains the full information of the true dictionary $\tilde{\bD}^*$, it is crucial for our learning algorithm to update our dictionary given the new dataset $\bA'$ while retaining our knowledge from the old dataset $\bA$. 

In our experiment, we first solve the standard dictionary learning problem in \eqref{eq:dict_learning_low} using the dataset $\bA$ to obtain the initial dictionary $\hat{\bD}$ and the coefficient vectors $\{\hat{\bx}_i\}_{i=1}^{250}$. Then we use the reconstruction error on $\bA$ with respect to $\{\hat{\bx}_i\}_{i=1}^{250}$ to define the lower-level objective in problem~\eqref{eq:dict_learning}, and use the error on the new dataset $\bA'$ to define the upper-level objective. In this case, $\hat{\bD}$ serves as a near-optimal solution for the lower-level problem. We compare CG-BiO with BiG-SAM and a-IRG.
{Similar to the previous experiment, we exclude MNG due to its computational intractability.}
Moreover, to demonstrate the necessity of the cutting plane in~\eqref{eq:X_k}, we also run a baseline method that performs the CG update over the set $\cZ$ instead of $\cX_k$ (cf. the update in~\eqref{eq:subproblem-lp}). This method ignores the lower-level objective and may be regarded as applying the standard CG method solely on the upper-level objective. In all algorithms, we initialize %
$\tilde{\bD}$ with the dictionary $\hat{\bD}$ learned from $\bA$ and initialize %
$\tilde{\bX}$ randomly.

We report our results in Fig.~\ref{fig:nonconvex}. 
In addition to the the upper- and lower-level objective values, we use the recovery rate of the true basis vectors as our performance metric. Specifically,  
a basis vector $\tilde{\bd}^*_i$ in $\tilde{\bD}^*$ is regarded as successfully recovered if there exists $\tilde{\bd}_j$ in $\tilde{\bD}$ such that $|\langle \tilde{\bd}^*_i, \tilde{\bd}_j\rangle|\!>\!0.9$. 
In Fig.~\ref{fig:nonconvex}(a) and (b), 
we observe that CG-BiO converges faster than BiG-SAM and a-IRG, and it also achieves smaller errors in terms of both the upper- and lower-level objectives. 
\blue{DBGD achieves a similar upper-level objective value as CG-BiO, but performs poorly in terms of the lower-level objective.}
On the other hand, the baseline CG method only focuses on the upper-level objective and as a result incurs a much larger error on the lower-level objective. 
In terms of recovery rate, Fig.~\ref{fig:nonconvex}(c) shows that CG-BiO recovers almost all basis vectors in $\tilde{\bD}^*$ at the end of its execution \blue{and performs slightly better than DBGD}. 
In contrast, both BiG-SAM and a-IRG only learn from the dataset $\bA$ due to their slow convergence, while the baseline CG method ``forgets'' the basis vectors previously learned and only recovers those underlying the new dataset~$\bA'$.

\section{CONCLUSION}
In this paper, we proposed a CG-based method to solve the class of simple bilevel optimization problems. We closed an important gap in the existing literature by providing a tight non-asymptotic complexity bound for the upper-level objective. Specifically, we proved that our CG-BiO method finds an $(\epsilon_f,\epsilon_g)$-optimal solution after at most ${\mathcal{O}}(\max\{1/\epsilon_f,1/\epsilon_g\})$ iterations when the upper-level objective $f$ is convex, and after at most ${\mathcal{O}}(\max\{1/\epsilon_f^2,1/(\epsilon_f\epsilon_g)\})$ iterations when $f$ is non-convex. 
We further strengthened our results when the lower-level problem satisfies the H\"olderian error bound assumption.
The numerical results also showed the superior performance of our method compared to existing algorithms. 

\subsubsection*{Acknowledgements}
The research of R. Jiang and A. Mokhtari is supported in part by NSF Grants 2127697, 2019844, and  2112471,  ARO  Grant  W911NF2110226,  the  Machine  Learning  Lab  (MLL)  at  UT  Austin, and the Wireless Networking and Communications Group (WNCG) Industrial Affiliates Program. The research of N. Abolfazli and E. Yazdandoost Hamedani is supported by NSF Grant 2127696.

\bibliographystyle{abbrvnat}
\bibliography{references}

\appendix
\onecolumn
\section{SUPPORTING LEMMAS}%
\subsection{Proof of Lemma~\ref{lem:subset}}
Let $\bx_g^*$ be any point in $\cX^*_g$, i.e., any optimal solution of the lower-level problem. By definition, we have $g(\bx_g^*) = g^*$. Since $g$ is convex and $g^*\leq g(\bx_0)$, we have 
\begin{equation*}
	g(\bx_0)-g(\bx_k) \geq g^*-g(\bx_k) = g(\bx_g^*)-g(\bx_k) \geq \fprod{\grad g(\bx_k),\bx_g^*-\bx_k},
\end{equation*}
which implies $\bx_g^*\in \cX_k$. Hence, we conclude that $\cX^*_g \subseteq \cX_k$. 

\subsection{Improvement in One Step}
The following lemma characterizes the improvement of both the upper-level and lower-level objective values after one step of Algorithm~\ref{alg:bilevel}.  
\begin{lemma}\label{lem:one-step}
	Let $\{\bx_k\}_{k=0}^{K}$ be the sequence generated by Algorithm~\ref{alg:bilevel}. Suppose Assumption \ref{assum:smooth} holds, then for any $k\geq 0$ we have 
	\begin{align}
		&f(\bx_{k+1})\leq f(\bx_k)-\gamma_k \FW(\bx_k) +\frac{1}{2}\gamma_k^2L_f D^2,\label{eq:one-step-f}\\
		&g(\bx_{k+1})\leq (1-\gamma_k) g(\bx_k)+\gamma_k g(\bx_0) +\frac{1}{2}\gamma_k^2 L_g D^2,\label{eq:one-step-g}
	\end{align}
\end{lemma}

\begin{proof}

Since the gradient of $f$ is $L_f$-Lipschitz and $\cZ$ is bounded with diameter $D$, we have 
\begin{align}
	f(\bx_{k+1}) &\leq f(\bx_k) + \fprod{\nabla f(\bx_k),\bx_{k+1}-\bx_k}+ \frac{1}{2}L_f \|\bx_{k+1}-\bx_k\|^2  \nonumber\\
	&=  f(\bx_k) + \gamma_k\fprod{\nabla f(\bx_k),\bs_{k}-\bx_k}+ \frac{1}{2}L_f\gamma_k^2\|\bs_k-\bx_k\|^2 \nonumber\\
	&\leq  f(\bx_k) + \gamma_k\fprod{\nabla f(\bx_k),\bs_{k}-\bx_k}+ \frac{1}{2}L_f\gamma_k^2D^2. \label{eq:one-step-f-1}
\end{align}
Now using the definition of $\bs_k$ in \eqref{eq:subproblem-lp}, the definition of $\FW(\bx)$ in \eqref{eq:FW_gap} and Lemma~\ref{lem:subset}, we obtain 
\begin{equation}\label{eq:Gamma_upper_bound}
	\fprod{\nabla f(\bx_k),\bs_{k}-\bx_k} = \min_{\bs \in \cX_k}~\fprod{\nabla f(\bx_k),\bs-\bx_k} \leq \min_{\bs \in \cX^*_g}~\fprod{\nabla f(\bx_k),\bs-\bx_k} = -\FW(\bx_k).
\end{equation}
Then \eqref{eq:one-step-f} follows from \eqref{eq:one-step-f-1} and \eqref{eq:Gamma_upper_bound}. 

Similarly, since the gradient of $g$ is $L_g$-Lipschitz, we have 
\begin{align}\label{eq:one-step-g-1}
	g(\bx_{k+1}) &\leq g(\bx_k) + \gamma_k\fprod{\nabla g(\bx_k),\bs_k-\bx_k}+ \frac{1}{2}L_g\gamma_k^2 D^2.
\end{align}
Moreover, since $\bs_k \in \cX_k$, from the definition of $\cX_k$ in \eqref{eq:subproblem-lp} we get $\fprod{\grad g(\bx_k), \bs_k-\bx_k} \leq g(\bx_0)-g(\bx_k)$. Combining this with \eqref{eq:one-step-g-1} leads to \eqref{eq:one-step-g}. 
\end{proof}

\section{PROOF OF THE MAIN THEOREMS}
\subsection{Proof of Theorem~\ref{thm:convex-upper-bound}}\label{appen:convex}
We first prove the convergence rate of the upper-level objective $f$, which largely mirrors the standard analysis of the CG method \citep{jaggi2013revisiting}. Since $\bx^*\in \cX_g^*$ and $f$ is convex, from the definition of $\FW(\bx_k)$ in \eqref{eq:FW_gap} we have 
\begin{align}\label{eq:conv-G}
    \FW(\bx_k) = \max_{\bs\in \cX^*_g}\{\fprod{\grad f(\bx_k),\bx_k-\bs}\} \geq \fprod{\grad f(\bx_k),\bx_k-\bx^*} \geq f(\bx_k)-f^*.
\end{align}
Subtracting $f^*$ from both sides of \eqref{eq:one-step-f} in Lemma~\ref{lem:one-step} and using \eqref{eq:conv-G}, we obtain that 
\begin{align}\label{eq:one_step_contraction}
f(\bx_{k+1})-f^*\leq (1-\gamma_k)(f(\bx_k)-f^*)+\frac{1}{2}\gamma_k^2 L_f D^2. %
\end{align}
Now define $A_k = k(k+1)$. By substituting $\gamma_k=2/(k+2)$ and multiplying both sides of \eqref{eq:one_step_contraction} by $A_{k+1}$, we get  
\begin{equation*}
  A_{k+1}(f(\bx_{k+1})-f^*) \leq A_k(f(\bx_{k})-f^*)+\frac{2(k+1)}{k+2}L_f D^2 \leq A_k(f(\bx_{k})-f^*)+ 2L_f D^2.
\end{equation*}
Hence, if follows from induction that 
\begin{equation*}
  A_K(f(\bx_K)-f^*) \leq A_0(f(\bx_0)-f^*)+ 2KL_f D^2 \quad \Rightarrow \quad  f(\bx_K)-f^* \leq \frac{2KL_f D^2}{A_k} = \frac{2L_f D^2}{K+1}.
\end{equation*}
This completes the first part of the proof. 

The proof for the lower-level problem follows from similar arguments. By subtracting $g(\bx_0)$ from both sides of \eqref{eq:one-step-g} in Lemma~\ref{lem:one-step}, we have 
\begin{equation}\label{eq:one_step_contraction_g}
  g(\bx_{k+1})-g(\bx_0)\leq (1-\gamma_k) (g(\bx_k)-g(\bx_0))+\frac{1}{2}\gamma_k^2 L_g D^2. 
\end{equation}
By substituting $\gamma_k = 2/(k+2)$ and multiplying both sides of \eqref{eq:one_step_contraction_g} by $A_{k+1}$, we obtain
\begin{equation*}
  A_{k+1}(g(\bx_{k+1})-g(\bx_0)) \leq A_k(g(\bx_{k})-g(\bx_0)) + 2L_g D^2.
\end{equation*}
Hence, if follows from induction that 
\begin{equation*}
  A_K(g(\bx_K)-g(\bx_0)) \leq 2KL_g D^2 \quad \Rightarrow \quad  g(\bx_K)-g(\bx_0) \leq \frac{2KL_g D^2}{A_k} = \frac{2L_g D^2}{K+1}.
\end{equation*}
Since $g(\bx_0)-g^*\leq \epsilon_g/2$, 
we obtain
\begin{equation*}
  g(\bx_K)-g^* \leq \frac{2L_g D^2}{K+1} + \frac{1}{2}\epsilon_g,
\end{equation*}
which completes the proof. 

\subsection{Proof of Theorem~\ref{thm:nonconvex-upper-bound}}\label{appen:nonconvex}
Since we use a fixed stepsize in Theorem~\ref{thm:nonconvex-upper-bound}, in the following we will write $\gamma_k=\gamma$. %

We first consider the upper-level objective $f$. The analysis here is similar to the one by \cite{mokhtari2018escape}. By using \eqref{eq:one-step-f} in Lemma~\ref{lem:one-step}, we have 
\begin{equation*}
	\FW(\bx_k) \leq \frac{f(\bx_k)-f(\bx_{k+1})}{\gamma}+\frac{1}{2}\gamma L_f D^2.
\end{equation*}
Summing both sides of the above inequality from $k=0$ to $K-1$, we get
\begin{equation*}
	\sum_{k=0}^{K-1} \FW(\bx_k) \leq \frac{f(\bx_0)-f(\bx_K)}{\gamma} + \frac{1}{2} K\gamma L_f D^2 \leq \frac{f(\bx_0)-\underf}{\gamma} + \frac{1}{2} K\gamma L_f D^2,
\end{equation*}
where we used the fact that $f(\bx_K) \geq \underf = \min_{\bx \in Z} f(\bx)$. This further implies that 
\begin{equation}\label{eq:FW_min}
	\min_{0\leq k \leq K-1}~\FW(\bx_k) \leq \frac{1}{K} \sum_{k=0}^{K-1} \FW(\bx_k) \leq \frac{f(\bx_0)-\underf}{\gamma K} + \frac{1}{2} \gamma L_f D^2.
\end{equation}
To upper bound the right-hand side of \eqref{eq:FW_min}, note that our choices of the stepsize $\gamma$ and the number of iterations $K$ satisfy 
\begin{equation*}
  \gamma \leq \frac{\epsilon_f}{ L_f D^2} \quad \text{and} \quad K \geq \frac{2(f(\bx_0)-\underf)}{\epsilon_f \gamma}.
\end{equation*}
Thus, we have 
\begin{equation*}
	\min_{0\leq k \leq K-1}~\FW(\bx_k) \leq \frac{f(\bx_0)-\underf}{\gamma K} + \frac{1}{2} \gamma L_f D^2 \leq \frac{\epsilon_f}{2} + \frac{\epsilon_f}{2} = \epsilon_f.
\end{equation*}
This guarantees that $\FW(\bx_{k^*}) \leq \epsilon_f$ by choosing $k^* = \argmin_{0\leq k \leq K-1}~\FW(\bx_k)$.

Now we move to the analysis of the lower-level objective $g$. For any $k\geq 0$, by applying induction on \eqref{eq:one-step-g} in Lemma~\ref{lem:one-step}, it follows that 
\begin{equation*}
  g(\bx_k) - g(\bx_0) \leq \frac{1}{2}L_g D^2 \sum_{j=0}^{k-1}\gamma^2 (1-\gamma)^j  \leq \frac{1}{2}L_g D^2 \gamma, 
\end{equation*}
where we used $\sum_{j=0}^{k-1} (1-\gamma)^j \leq 1/\gamma$ in the last inequality. Furthermore, since $g(\bx_0)-g^* \leq \epsilon_g/2$ and $\gamma\leq \frac{\epsilon_g}{L_g D^2}$, this implies that 
$g(\bx_k) - g^* \leq \frac{1}{2} \epsilon_g + \frac{1}{2} \epsilon_g = \epsilon_g$ for any $0\leq k \leq K-1$. In particular, we can take $k=k^*$ and conclude that $g(\bx_{k^*})-g^* \leq \epsilon_g$. This completes the proof. 

\section{PROOFS UNDER H\"OLDERIAN ERROR BOUND ASSUMPTION}\label{appen:error_bound}

\subsection{Proof of Proposition~\ref{prop:error_bound}}
Since $\cX_g^*$ is closed and compact, we can let $\hat{\bx}^* = \argmin_{\bx\in \cX_g^*} \|\bx-\hat{\bx}\|$ such that $\|\hat{\bx}^*-\hat{\bx}\|= \mathrm{dist}(\hat{\bx}, \cX^*_g)$. By Assumption~\ref{assum:Holder}, we obtain 
\begin{equation*}
  \frac{\alpha}{r} \|\hat{\bx}^*-\hat{\bx}\|^r \leq g(\hat{\bx})-g^* \leq \epsilon_g \quad \Leftrightarrow \quad \|\hat{\bx}^*-\hat{\bx}\| \leq \left(\frac{r\epsilon_g}{\alpha}\right)^{\frac{1}{r}}.
\end{equation*}
When $f$ is convex, we have 
\begin{equation*}
  f(\hat{\bx}) -f^* \geq f(\hat{\bx}) -f(\hat{\bx}^*) \geq \langle \grad f(\hat{\bx}^*), \hat{\bx}- \hat{\bx}^* \rangle \geq -\|\grad f(\hat{\bx}^*)\|_* \|\hat{\bx}- \hat{\bx}^*\| \geq -M \left(\frac{r\epsilon_g}{\alpha}\right)^{\frac{1}{r}},
\end{equation*}
where we used the convexity of $f$ in the second inequality. 
When $f$ is non-convex, we have 
\begin{align}
  \FW(\hat{\bx}) = \max_{\bs\in \cX^*_g}\{\fprod{\grad f(\hat{\bx}),\hat{\bx}-\bs}\} &\geq  \fprod{\grad f(\hat{\bx}),\hat{\bx}-\hat{\bx}^*} \nonumber\\
  & = \fprod{\grad f(\hat{\bx})- \grad f(\hat{\bx}^*), \hat{\bx}-\hat{\bx}^*} + \fprod{\grad f(\hat{\bx}^*), \hat{\bx}-\hat{\bx}^*} \nonumber\\
  & \geq -\|\grad f(\hat{\bx})- \grad f(\hat{\bx}^*)\|_*\|\hat{\bx}-\hat{\bx}^*\| - \|\grad f(\hat{\bx}^*)\|\|\hat{\bx}-\hat{\bx}^*\| \nonumber\\
  & \geq -L_f\|\hat{\bx}-\hat{\bx}^*\|^2- M\|\hat{\bx}-\hat{\bx}^*\| \label{eq:using_smooth}\\
  & \geq  -M \left(\frac{r\epsilon_g}{\alpha}\right)^{\frac{1}{r}}- L_f\left(\frac{r\epsilon_g}{\alpha}\right)^{\frac{2}{r}},\nonumber
\end{align}
where we used the fact that $\nabla f$ is $L_f$-Lipschitz in \eqref{eq:using_smooth}. 
This completes the proof. 

\subsection{Proof of Corollary~\ref{coro:stronger_complexity}}
In the first case where $f$ is convex, we set $\epsilon_g=\frac{\alpha}{r}  \left(\frac{\epsilon_f}{M}\right)^r$. By Theorem~\ref{thm:convex-upper-bound}, we have $f(\bx_K)-f^* \leq \epsilon_f$ and $g(\bx_K)-g^* \leq \epsilon_g$ when 
\begin{equation*}
  K \geq \max\biggl\{\frac{2L_f D^2}{\epsilon_f}-1,\frac{4L_g D^2}{\epsilon_g}-1\biggr\} = 
  \max\biggl\{\frac{2L_f D^2}{\epsilon_f}-1,\frac{4r M^r L_g D^2}{\alpha \epsilon_f^{r}}-1\biggr\}= \mathcal{O}\biggl(\frac{1}{\epsilon_f^r}\biggr).
\end{equation*}
Moreover, Proposition~\ref{prop:error_bound} implies that $f(\bx_K)-f^* \geq -M \left(\frac{r\epsilon_g}{\alpha}\right)^{\frac{1}{r}}\geq -\epsilon_f$. Putting all pieces together, we conclude that $|f(\bx_K)-f^*|\leq \epsilon_f$ and $g(\bx_K)-g^* \leq \epsilon_g$ after $K = \bigO({1}/{\epsilon_f^r})$ iterations. 

In the second case where $f$ is non-convex, we set $\epsilon_g=\min\{\frac{\alpha}{r}  \left(\frac{\epsilon_f}{2M}\right)^r, \frac{\alpha}{r}\bigl(\frac{ \epsilon_f}{2L_f}\bigr)^{r/2}\}$. By Theorem~\ref{thm:nonconvex-upper-bound}, we can find $k^*\in \{0,1,\dots,K-1\}$ such that $
\FW (\bx_{k^*})\leq \epsilon_f
$ and $g(\bx_{k^*})-g^*\leq \epsilon_g$ when 
\begin{align*}
  K &\geq (f(\bx_0)-\underf)\cdot \max\biggl\{\frac{2L_fD^2}{\epsilon_f^2},\frac{2L_g D^2 }{\epsilon_f\epsilon_g }\biggr\}\\
  &= (f(\bx_0)-\underf)\cdot\max\biggl\{\frac{2L_fD^2 }{\epsilon_f^2},\frac{2r(2M)^rL_g D^2 }{\alpha \epsilon_f^{r+1} },\frac{2r(2L_f)^{\frac{r}{2}}L_g D^2 }{\alpha \epsilon_f^{\frac{r}{2}+1}}\biggr\} = \bigO\biggl(\frac{1}{\epsilon_f^{r+1}}\biggr).
\end{align*}
Moreover, Proposition~\ref{prop:error_bound} implies that $\FW(\bx_{k^*})\geq -M \left(\frac{r\epsilon_g}{\alpha}\right)^{\frac{1}{r}}- L_f\left(\frac{r\epsilon_g}{\alpha}\right)^{\frac{2}{r}} \geq -\frac{\epsilon_f}{2}-\frac{\epsilon_f}{2} = -\epsilon_f$. Thus, we conclude $|\FW(\bx_{k^*})|\leq \epsilon_f$ and $g(\bx_{k^*})-g^* \leq \epsilon_g$ after $K = \bigO({1}/{\epsilon_f^{r+1}})$ iterations. 

\section{PRIMAL-DUAL METHOD FOR THE BILEVEL PROBLEM}\label{appen:primal-dual}
In this section, we discuss the convergence rate of primal-dual type methods for solving the bilevel problem in \eqref{eq:bi-simp}. 
{We consider the setting as in Theorem~\ref{thm:convex-upper-bound}, in which both $f$ and $g$ are convex and smooth.} 
To simplify the discussion, we further assume $\cZ=\{\bz\in\cX\mid \bA\bz\leq \bb\}$ where $\bA\in\reals^{m\times d}$, $\bb\in\reals^m$, and $\cX$ is a convex and easy-to-project compact set.

To obtain the reformulation in \eqref{eq:constrained_formulation}, one first needs to estimate the optimal value $g^*$ of the lower-level problem. %
Since it is a convex program with linear constraints, we can implement a first-order primal-dual method (see, e.g., \citep{chambolle2016ergodic}) to find $g_0$ such that $\abs{g_0-g^*}\leq \epsilon_g/4$ within at most $\cO(\frac{L_g+\norm{\bA}}{\epsilon_g})$ iterations\footnote{Note that this complexity can be improved to {the optimal rate of} $\cO({\sqrt{\frac{L_g}{\epsilon_g}}+\frac{\norm{\bA}}{\epsilon_g}})$ using an accelerated method. %
}. Next, problem \eqref{eq:bi-simp} can be cast as the following convex optimization problem with linear and nonlinear convex constraints:
\begin{equation}\label{eq:constrained_relaxed}
  \min_{\bx\in \cX}~f(\bx)\quad \hbox{s.t.}\quad  \bA\bx\leq \bb,\,g(\bx) \leq g_0+{\frac{\epsilon_g}{2}},
\end{equation}
{where we add the term $\frac{\epsilon_g}{2}$ to ensure that the Slater's condition holds}. 
Now we can apply any classic or accelerated first-order primal-dual methods \citep{he2015mirror,xu2021iteration,hamedani2021primal} to find a solution of problem~\eqref{eq:constrained_relaxed} that is both $\epsilon_f$-suboptimal and $\frac{\epsilon_g}{4}$-infeasible. For example, the optimal convergence rates obtained by \cite{xu2021iteration} and \cite{hamedani2021primal}
imply that after $K$ iterations, the average iterate $\bar{\bx}_K$ satisfies 
\begin{equation*}
  \max\left\{\abs{f(\bar \bx_K)-f(\bx^*_\epsilon)},\abs{g(\bar\bx_K)-g(\bx^*_\epsilon)}\right\}\leq \Delta/K,
\end{equation*}
where $\bx^*_\epsilon$ denotes an optimal solution of problem \eqref{eq:constrained_relaxed}, $\Delta\triangleq \cO((L_f+L_g+C_g)D^2+C_g\abs{\lambda_1^*}^2+\norm{\bA}\norm{\blambda_2^*}^2)$, {$C_g$ is the Lipschtiz constant of $g$}, and $\lambda_1^*\in\reals$ and $\blambda_2^*\in\reals^{m}$ denote an arbitrary dual optimal solution corresponding to the nonlinear and linear constraints in problem~\eqref{eq:constrained_relaxed}, respectively. 
Using the fact that $f(\bx^*_\epsilon) \leq f(\bx^*)$ and $g(\bx^*_\epsilon)\leq g_0+\frac{\epsilon_g}{2}\leq g^*+\frac{3}{4}\epsilon_g$, we conclude
\begin{align*}
    f(\bar \bx_K)-f(\bx^*)\leq \Delta/K\quad \text{and} \quad%
    \abs{g(\bar\bx_K)-g(\bx^*)}\leq \Delta/K+\frac{3}{4}\epsilon_g. %
\end{align*}
Therefore, to achieve an $(\epsilon_f,\epsilon_g)$-optimal solution of problem~\eqref{eq:bi-simp}, %
a primal-dual method overall requires $\cO\left(\frac{L_g+\norm{\bA}}{\epsilon_g}+\frac{\Delta}{\min\{\epsilon_f,\epsilon_g\}}\right)$ primal-dual gradient calls, whereas our proposed method overall requires $\cO\left(\frac{L_g}{\epsilon_g}+\frac{(L_f+L_g)D^2}{\min\{\epsilon_f,\epsilon_g\}}\right)$ linear minimization oracle calls. 
{In particular, we observe that the convergence guarantee of primal-dual methods heavily relies on the norm of the dual optimal variable $|\lambda_1^*|$, which may tend to infinity as $\epsilon$ approaches zero and the problem in \eqref{eq:constrained_relaxed} becomes nearly degenerate.}

\vspace{-1mm}
\subsection{Numerical Example}\label{subsec:toy}
\vspace{-1mm}

Here we consider a simple two-dimensional example to illustrate the numerical instability of primal-dual methods applied to the relaxed problem \eqref{eq:constrained_relaxed}. To this end, consider the following problem
\begin{align}\label{eq:toy}
    \min_{\bx\in\reals^2}0.5x_1^2-0.5x_1+0.1x_2\quad \hbox{s.t.} \quad \bx\in\argmin_{\bz\in\cZ}\{-z_1-z_2\},
\end{align}
where $\cZ=\{\bz\in\reals^2_+\mid z_1+z_2\leq 1, 4z_1+6z_2\leq 5\}$. 
The lower-level problem has multiple solutions which can be described by $\cX^*_g=\{\bx\in\reals^2\mid x_1+x_2=1, x_1\in[0.5,1], x_2\in[0,0.5]\}$ and the optimal solution of \eqref{eq:toy} is $(x_1^*,x_2^*)=(0.6,0.4)$. We implemented 
accelerated primal-dual method with backtracking (APDB) proposed by \cite{hamedani2021primal}, one of the state-of-the-art primal-dual methods, and 
compared it with our proposed method CG-BiO. Figure \ref{fig:toy} illustrates the iteration {trajectories} of both methods. We selected the relaxing parameter in \eqref{eq:constrained_relaxed} as $\epsilon=10^{-5}$ for APDB. We also used the same accuracy for $\epsilon_g$ and $\epsilon_f$ when implementing CG-BiO. The primal-dual method finds an $\epsilon$-solution (dark red cross) within 193 iterations while CG-BiO finds an $\epsilon$-solution (green star) within 20 iterations. Furthermore, we observe a more stable numerical behavior for CG-BiO in comparison with APDB, which corroborates our theoretical analysis above. %

\begin{figure*}
    \centering
    \subfloat[Lower-level gap]{\includegraphics[scale=0.25]{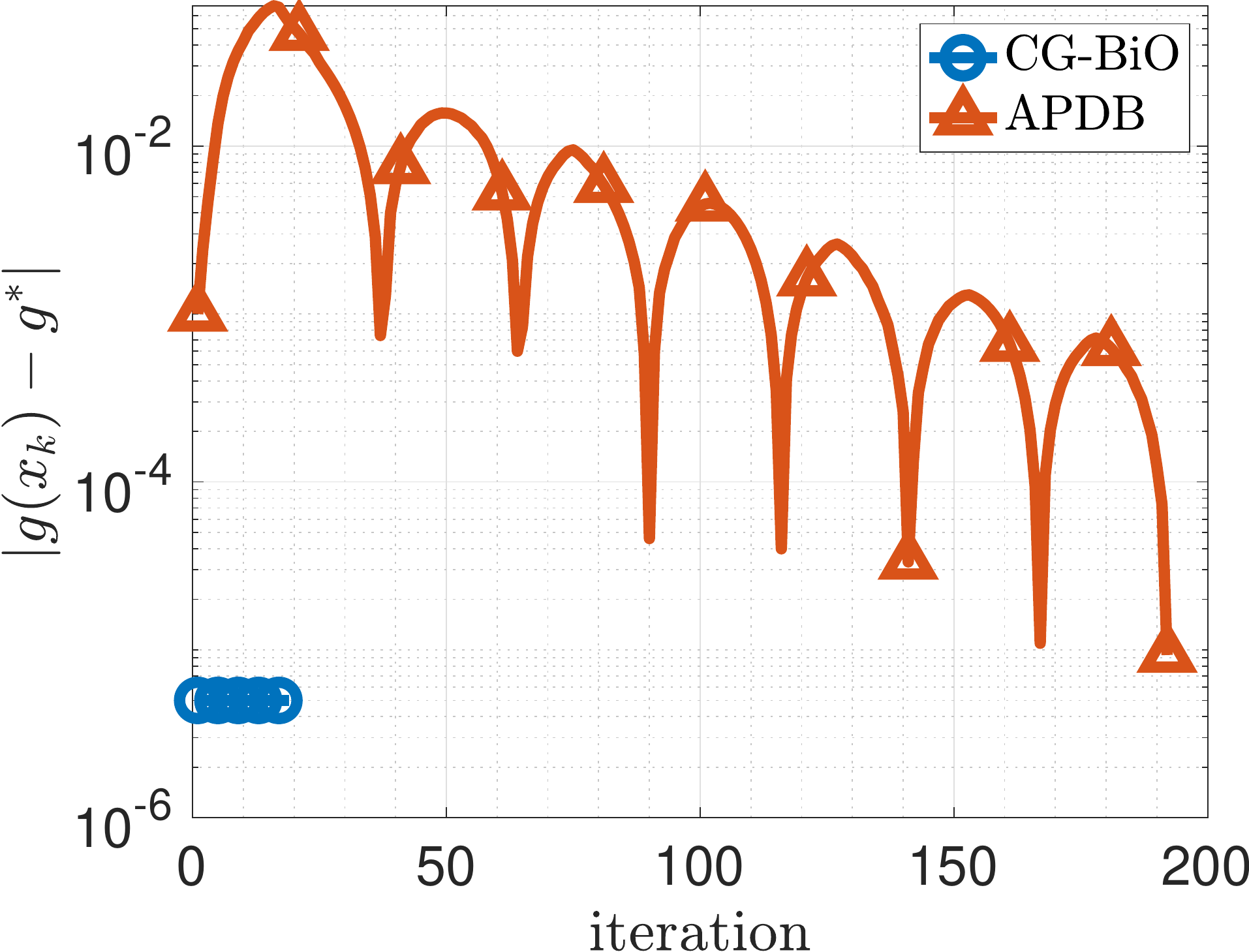}}
    \quad
    \subfloat[Upper-level gap]{\includegraphics[scale=0.25]{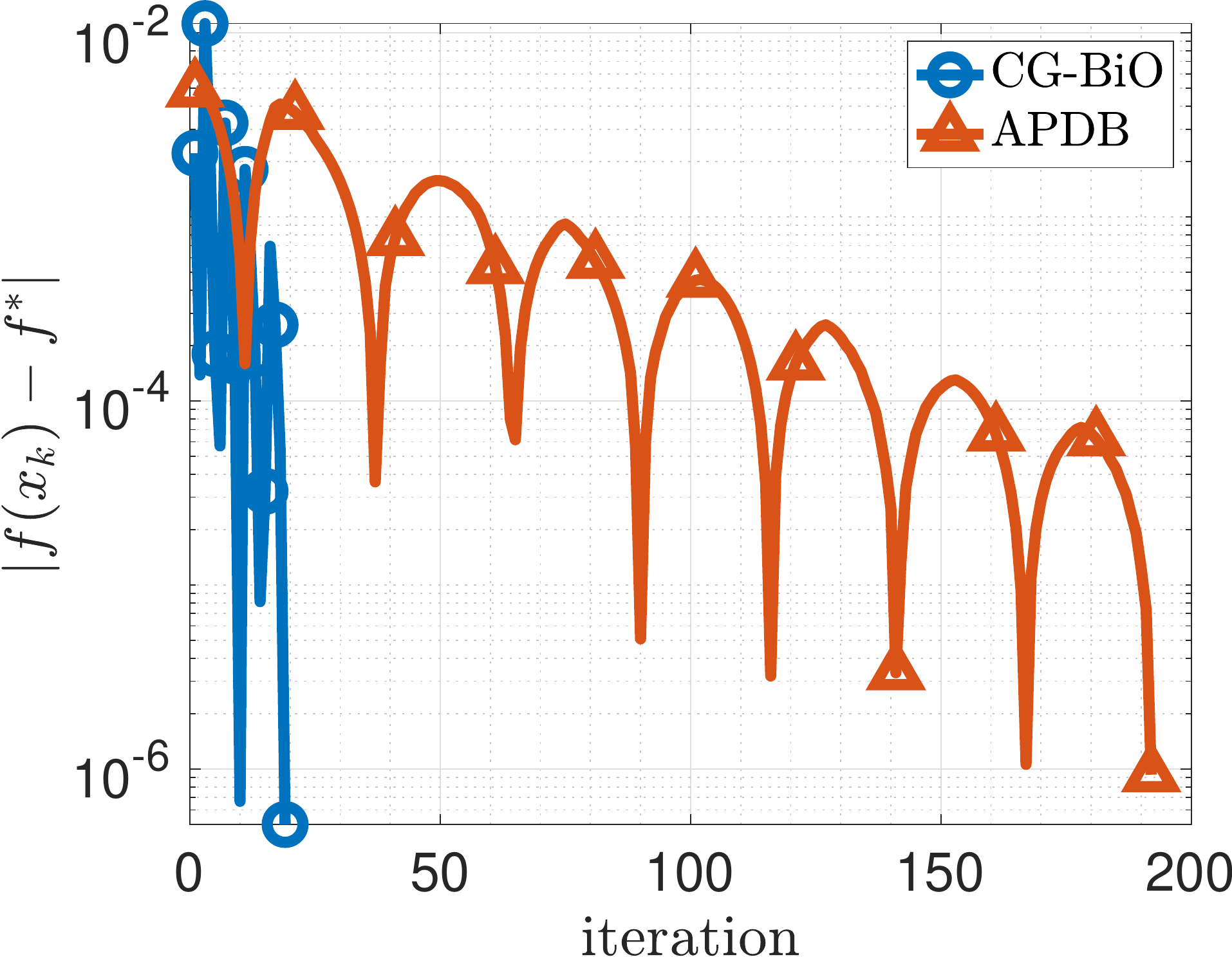}}
    \quad
    \subfloat[Iteration trajectory]{\includegraphics[scale=0.25]{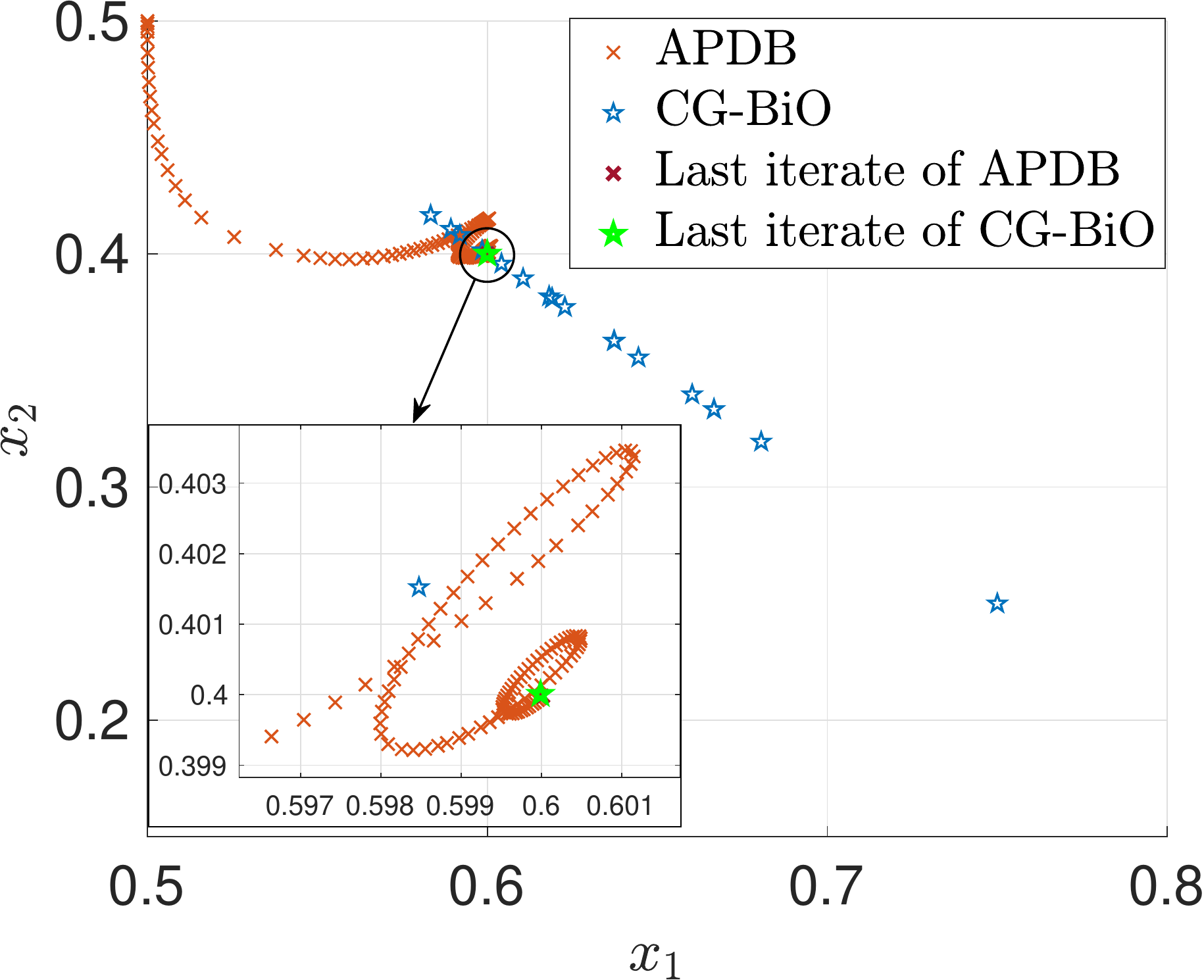}}
    \caption{The performance of CG-BiO compared with APDB on problem \eqref{eq:toy}.}
    \label{fig:toy}
\end{figure*}

\section{ADDITIONAL MOTIVATING EXAMPLES}\label{sec:additional_example}
In this section,
we provide some additional remarks and two more examples for the bilevel problem in \eqref{eq:bi-simp}. 

\subsection{{Lexicographic Optimization}}
In Section~\ref{subsec:examples}, we have seen two instances of lexicographic optimization, where we use the secondary loss to improve generalization (Example~\ref{ex:over_regression}) or promote fairness (Example~\ref{ex:fair_class}). In the following, we describe another standard example where we use regularization to tackle ill-posed problems. 

\begin{example}[Ill-posed Optimization] 
Without an explicit regularization, the empirical risk minimization problem $\min_{\bz\in \mathcal{Z}} \ell_{\mathrm{tr}}(\bz)$ can be ill-posed, i.e., it has multiple optimal solutions or is sensitive to small perturbation in the input data.
To tackle this issue, we can consider 
a secondary objective function $\cR(\cdot)$ as another criterion to select one of the optimal solutions with some desired property. 
For example, 
we can find the minimal $\ell_2$-norm solution by choosing $\cR(\bx)=\frac{1}{2}\norm{\bx}_2^2$. 
Such a problem can be formulated as the following bilevel problem:
\begin{equation*}
    \min_{\bbeta\in \reals^d}~\cR(\bbeta)\quad \hbox{s.t.}\quad  \bbeta\in\argmin_{\bz\in \cZ}~\ell_{\mathrm{tr}}(\bz).
\end{equation*} 
\end{example}

\subsection{Lifelong Learning}

{
A popular framework known as A-GEM~\citep{chaudhry2018efficient} formulates the lifelong learning problem as follows: 
\begin{equation}\label{eq:aGEM}
\min_{\bbeta}~\frac{1}{n'}\sum_{i=1}^{n'}\ell( \langle \bx'_i, \bbeta \rangle,y'_i)\quad \hbox{s.t.}\quad  \sum_{(\bx_i,y_i)\in \mathcal{M}}\ell( \langle \bx_i, \bbeta \rangle,y_i) \leq \sum_{(\bx_i,y_i)\in \mathcal{M}}\ell( \langle \bx_i, \bbeta^{(t-1)} \rangle,y_i). 
\end{equation}
Here, the objective function is the training loss on the current task $\mathcal{D}_t=\{(\bx'_i,y_i')\}_{i=1}^{n'}$, and the inequality constraint ensures that the model with parameter $\bbeta$ performs no worse than the previous one on the episodic memory $\mathcal{M}$ (i.e., some stored data samples from the previous tasks). 
}

{
In this paper, we consider a variant of problem \eqref{eq:aGEM}, where we further tighten the constraint and require that the model also minimizes the error on the episodic memory.
This leads to the following bilevel problem: 
\begin{equation}\label{eq:agem-bilevel}
    \min_{\bbeta}~\frac{1}{n'}\sum_{i=1}^{n'}\ell( \langle \bx'_i, \bbeta \rangle,y'_i)\quad \hbox{s.t.}\quad  \bbeta \in \argmin_{\bz}~\sum_{(\bx_i,y_i)\in \mathcal{M}}\ell( \langle \bx_i, \bz \rangle,y_i).
\end{equation}
Example~\ref{ex:dict} can be viewed as an instance of problem~\eqref{eq:agem-bilevel} where the learning problem at hand is dictionary learning. Below we present another related example from representation learning. 
}

\begin{example}[Representation Learning] In meta-learning problems, we aim to pre-train a model that can be easily fine-tuned to new tasks. This can be often achieved by learning a compact representation that is shared among multiple tasks~\citep{tripuraneni2021provable,du2021fewshot,collins2022maml}. 
In particular, consider a multi-task linear representation learning problem with $T$ tasks at the training time. We assume that the data points for the $i$-th task are generated according to $y_i^j= {\bw_i^*}^\top {\bB^*}^\top \bx_i^j+n_i^j$ for $j=1,\dots,m_i$, where $n_i^j$ is some random noise and $\bB^*\in \reals^{k\times d}$ is a common representation that maps the input in $\mathbb{R}^d$ to a lower dimensional feature vector in $\mathbb{R}^k$. When we have access to a diverse set of tasks such that their heads $\{\bw_i^*\}_{i=1}^T$ span $\mathbb{R}^k$, it is shown that one can find the ground truth representation  $\bB^*$ by solving the following problem: %
  $$
  \min_{\bB}\min_{\bw_1,\dots,\bw_T}\sum_{i=1}^T  \sum_{j=1}^{m_i} \left( y_i^{j}\!-\! \bw_i^{ \top} \bB^{ \top} \bx_i^{j}\right)^2 \qquad \hbox{s.t.}\quad \|\bB\|_F\leq \Delta,\; \|\bw_i\|_1\leq \delta,\;i=1,\dots,T,
  $$
  where we impose the norm constraints on $\bB$ and $\{\bw_i\}_{i=1}^T$ for some parameters $\Delta,\delta>0$ to resolve the scale invariance of the problem. 

  However, if the tasks at the training time are not diverse enough, then we can only learn a partial represention, i.e., a subset of the feature maps in $\bB^*$. 
  One way to further improve the learned representation is to leverage the new tasks we observe during the test time. 
  Concretely, let $\hat{\bw}_1^*,\dots, \hat{\bw}_T^*$ and $\hat{\bB}_{\mathrm{tr}}^*$ denote the output of the training procedure.
  When we are given a new task at the test time, we can improve the representation $\hat{\bB}_{\mathrm{tr}}^*$ by solving the following bilevel problem:
  \begin{equation}\label{eq:representation_bilevel}
     \min_{\bB\in \reals^{k\times d}}\min_{\bw_{T+1}\in \reals^k}  f(\bB,\bw_{T+1})
     \qquad\text{s.t.} \quad \bB\in  \argmin_{\|\bB'\|_F\leq \Delta} g(\bB),\;\|\bw_{T+1}\|_1\leq \delta,
  \end{equation}
  where $f(\bB,\bw_{T+1})\triangleq \sum_{j=1}^{m_{T+1}} \bigl( y_{T+1}^{j}\!-\! \bw_{T+1}^{ \top} \bB^{ \top} \bx_{T+1}^{j}\bigr)^2$ is the loss over the test set and $g(\bB) \triangleq \sum_{i=1}^T  \sum_{j=1}^{m_i} \bigl( y_i^{j}\!-\! {\hat{\bw}_i}^{*\top}{\bB}^{ '\top} \bx_i^{j}\bigr)^2$ is the loss over the training set. 
  The rationale is that 
  the solution of problem~\eqref{eq:representation_bilevel} can fit to both the old training tasks and the new test task, and hence is a better approximation of $\bB^*$ compared to $\hat{\bB}_{\mathrm{tr}}^*$. This way, 
  we maintain the feature maps learned at the training time and at the same time learn new feature maps from the test task. 
  Note that in problem~\eqref{eq:representation_bilevel} the upper-level function is nonconvex, while the lower-level problem is convex with multiple solutions. 
  \end{example}

\section{EXPERIMENT DETAILS}\label{appen:experiment}
In this section, we include more details of the numerical experiments in Section~\ref{sec:numeric}. 

For completeness, we briefly review the update rules of MNG~\citep{beck2014first}, BiG-SAM~\citep{sabach2017first}, a-IRG~\citep{Kaushik2021} \blue{and DBGD~\citep{gong2021automatic}} in the setup of problem~\eqref{eq:bi-simp}. In the following, we use $\Pi_{\cZ}(\cdot)$ to denote the Euclidean projection onto the set $\cZ$.

\begin{itemize}
\item {
Each step of MNG requires solving the following subproblem: 
\begin{equation}\label{eq:MNG}
    \bx_{k+1} = \argmin_{\bx\in Q_{k} \cap W_{k}} f(\bx),
\end{equation}
where 
\begin{align*}
    Q_k &\triangleq \left\{\bz\in \reals^d\mid \langle G_M(\bx_k), \bx_k-\bz \rangle \geq \frac{3}{4M}\|G_M(\bx_{k})\|^2\right\}, \\
    W_k &\triangleq \left\{\bz\in \reals^d\mid \langle \nabla f(\bx_k),\bz-\bx_k \rangle \geq 0 \right\}, \\
    G_M(\bx) &\triangleq M\left[\bx-\Pi_{\mathcal{Z}}\left(\bx-\frac{1}{M}\nabla g(\bx)\right)\right],
\end{align*}
and $M \geq L_g$ is a hyperparameter. As we can see, the implementation of MNG is only feasible when the subproblem in \eqref{eq:MNG} is easy to solve. In particular, it is typically computationally intractable when the upper-level objective $f$ is non-convex. 
}
\item
BiG-SAM is given by 
\begin{align*}
    \by_{k+1} &= \Pi_{\mathcal{Z}}(\bx_k-\eta_g \grad g(\bx_k)), \\
    \bz_{k+1} &= \bx_k - \eta_f\grad f(\bx_k), \\
    \bx_{k+1} &= \alpha_{k+1} \bz_{k+1} + (1-\alpha_{k+1}) \by_{k+1},
\end{align*}
where $\eta_f\leq \frac{2}{L_f}$ and $\eta_g \leq \frac{1}{L_g}$ are stepsizes and $\alpha_k = \min\{\frac{\gamma}{k},1\}$ for some $\gamma>0$.
{We note that the analysis by \cite{sabach2017first} requires the upper-level objective to be strongly convex, and therefore is not directly applicable in our setting. Nevertheless, we also implement their method and manually set the hyperparameters.}

\item
The a-IRG algorithm is given by
\begin{equation*}
    \bx_{k+1} = \Pi_{\mathcal{Z}}\left( \bx_k - \gamma_k(\grad g(\bx_k)+\eta_k \grad f(\bx_k))\right),
\end{equation*}
where $\gamma_k$ is the stepsize and $\eta_k$ is the regularization parameter. In our experiment, we choose $\gamma_k =  \gamma_0/\sqrt{k+1}$ and $\eta_k=\eta_0/(k+1)^{1/4}$ for some constants $\gamma_0$ and $\eta_0$. 

\item \blue{The DBGD algorithm is given by 
\begin{equation}\label{eq:DBGD_unconstrained}
    \bx_{k+1} = \bx_k - \gamma_k(\grad f (\bx_k)+\lambda_k \grad g (\bx_k)),
\end{equation}
where $\gamma_k$ is the stepsize and we set $\lambda_k$ as 
\begin{equation*}
    \lambda_k = \max\left\{\frac{\phi(\bx_k)-\langle \grad f(\bx_k), \grad g(\bx_k) \rangle}{\|\grad g(\bx_k)\|^2},0\right\} \quad \text{and} \quad \phi(\bx) = \min\{\alpha (g(\bx)-\hat{g}), \beta \|\grad g (\bx)\|^2\}.
\end{equation*}
Here, $\alpha$ and $\beta$ are hyperparameters and $\hat{g}$ is a lower bound on $g^*$. In our experiment, we choose $\alpha=\beta=1$ and $\hat{g}=0$. We also note that \cite{gong2021automatic} only considered unconstrained simple bilevel problems where $\mathcal{Z}=\reals^d$. To enforce the constraint, we replace \eqref{eq:DBGD_unconstrained} with the update rule $\bx_{k+1} = \Pi_{\mathcal{Z}}(\bx_k - \gamma_k(\grad f (\bx_k)+\lambda_k \grad g (\bx_k)))$. }
\end{itemize}

\subsection{Over-parameterized Regression}\label{appen:overparameterized}

\noindent\textbf{Dataset Generation.} The original Wikipedia Math Essential dataset~\citep{rozemberczki2021pytorch} consists of an 1068$\times$731 matrix. We randomly select one of the columns as the outcome vector $\bb\in \reals^{1068}$ and the rest as the data matrix $\bA\in \reals^{1068\times 730}$. We let $\lambda=1$ in the experiment, i.e., the constraint set is given by $\cZ=\{\bbeta\mid \|\bbeta\|_1 \leq 1\}$. 

\noindent\textbf{Initialization.} We run the standard CG method with the stepsizes chosen as $2/(k+2)$ on the lower-level problem in \eqref{eq:over_regression}. We terminate the procedure once the FW gap is no more than $\epsilon_g/2 = 5\times 10^{-5}$ or we have reached the maximum number of iterations $N_\mathrm{max}=10^4$.  

\noindent\textbf{Implementation Details.} 
For our CG-BiO method, we set the target accuracies for the upper-level and lower-level problems to {$\epsilon_f = 10^{-4}$} and $\epsilon_g = 10^{-4}$, respectively. We choose the stepsizes as $\gamma_k = 2/(k+12)$ to avoid instability due to large initial stepsizes. 
{In each iteration, we need to solve a subproblem in the form of
\begin{equation}\label{eq:LP}
    \min_{\bs}~\langle \grad f(\bbeta_k),\bs\rangle\qquad\text{s.t.} \quad \|\bs\|_1\leq \lambda,~\langle \grad g(\bbeta_k),\bs-\bbeta_k \rangle\leq  g(\bbeta_0)-g(\bbeta_k).
\end{equation}
We can reformulate the above problem as a linear program by introducing $\bs^{+},\bs^{-}\geq 0$ such that $\bs = \bs^{+}-\bs^{-}$. Specifically, problem~\eqref{eq:LP} becomes
\begin{align*}
    \min_{\bs^{+},\bs^{-}}~&\langle \grad f(\bbeta_k),\bs^{+}-\bs^{-}\rangle\\
    \text{s.t.} \quad &\bs^{+},\bs^{-}\geq 0,~\langle \bs^{+}, \mathbf{1}\rangle+\langle \bs^{-}, \mathbf{1}\rangle \leq \lambda,~\langle \grad g(\bbeta_k),\bs^{+}-\bs^{-}-\bbeta_k \rangle\leq  g(\bbeta_0)-g(\bbeta_k),
\end{align*}
where $\mathbf{1}\in \reals^d$ is the all-one vector. 

For MNG, we set $M=\lambda_{\mathrm{max}}(\bA_{\mathrm{tr}}^\top \bA_{\mathrm{tr}})$. For BiG-SAM, we set $\eta_f=2/\lambda_{\mathrm{max}}(\bA_{\mathrm{val}}^\top \bA_{\mathrm{val}})$, $\eta_g = 1/\lambda_{\mathrm{max}}(\bA_{\mathrm{tr}}^\top \bA_{\mathrm{tr}})$ and $\gamma=10$. For a-IRG, we set $\gamma_0=0.01$ and $\eta_0 = 1$. \blue{For DBGD, we set $\gamma_k = 10^{-4}$.}
}

\subsection{Fair Classification}
\noindent\textbf{Dataset Generation.} We preprocess the original Adult income dataset~\citep{Dua_2019} with the same procedure as in \citep{zafar2017fairness}, leading to a dataset with 50 attributes for prediction.  Moreover, we standardize all the attributes such that they lie between 0 and 1. In our experiment, we set $\lambda=100$.

\noindent\textbf{Initialization.} We run the standard CG method with backtracking line search~\citep{pedregosa20a} on the sparse logistic regression problem in \eqref{eq:fair_lower_level}. We terminate the procedure once the FW gap is no more than $\epsilon_g/2 = 5\times 10^{-5}$ or we have reached the maximum number of iterations $N_\mathrm{max}=10^4$.

\noindent\textbf{Implementation Details.}
For our CG-BiO method, we set {$\epsilon_f = 10^{-4}$} and $\epsilon_g = 10^{-4}$, respectively. We choose the stepsize as $\gamma_k = 0.005/\sqrt{k+1}$ instead of a constant stepsize as suggested by Theorem~\ref{thm:nonconvex-upper-bound}. Empirically, we observe that this leads to faster convergence. The subproblem we need to solve is in the same form as problem~\eqref{eq:LP}, which is also solved by a LP solver. 

For BiG-SAM, we set $\eta_f=\eta_g=0.1$ and $\gamma =1$. For a-IRG, we set $\gamma_0=5$ and $\eta_0=0.1$. \blue{For DBGD, we set $\gamma_k = 0.08$.}

\subsection{Dictionary Learning}\label{appen:dict_learning}
\noindent\textbf{Dataset Generation.}
Each of the sparse coefficient vectors $\{\bx_i\}_{i=1}^{250}$ and $\{\bx_k'\}_{k=1}^{200}$ has 5 nonzero entries, whose locations are randomly chosen. Also, the absolute values of those nonzero weights are drawn uniformly from the interval $[0.2,1]$. The entries of the random noise vectors $\{\bn_i\}_{i=1}^{250}$ and $\{\bn'_k\}_{k=1}^{200}$ follow i.i.d.\ Gaussian distribution with mean 0 and standard deviation 0.01.

\noindent\textbf{Initialization.} %
The initialization consists of two phases. In the first phase, we run the standard CG algorithm on both the variables $\bD\in \reals^{25\times 40}$ and $\bX \in \reals^{40\times 250}$ for $10^4$ iterations with the stepsize chosen as $1/\sqrt{k+1}$ ($k\geq 0$ is the iteration counter). Then in the second phase, we keep the variable $\bX$ fixed and update $\bD$ using the standard CG algorithm with exact line search. 
We terminate the procedure and output $\hat{\bD}$ and $\hat{\bX}$ until the FW gap with respect to $\bD$ is no more than $\epsilon_g=10^{-6}$. 

\noindent\textbf{Implementation Details.} We choose $\delta = 3$ in both problems~\eqref{eq:dict_learning_low} and~\eqref{eq:dict_learning}. 
All three algorithms start from the same initial point. We initialize $\tilde{\bD}\in \reals^{25\times 50}$ as the concatenation of $\hat{\bD}\in \reals^{25\times 40}$ and 10 columns of all zeros. Moreover, we initialize the variable $\tilde{\bX}$ randomly by drawing its entries from a standard Gaussian distribution and 
normalizing each column to have a $\ell_1$-norm of $\delta$. 
For our CG-BiO method, we choose the stepsize as $\gamma_k = 0.3/\sqrt{k+1}$ instead of a constant stepsize as suggested by Theorem~\ref{thm:nonconvex-upper-bound}. Empirically, we observe that this leads to faster convergence. The same stepsize rule is also used in the baseline CG method. %
{In each iteration, we need to solve a subproblem in the form of
\begin{equation}\label{eq:dict_sub}
    \min_{\tilde{\bD}}~\langle \grad f_{\tilde{\bD}}(\tilde{\bD}_k,\tilde{\bX}_k),\tilde{\bD}\rangle\qquad\text{s.t.} \quad \|\tilde{\bd}_i\|_2\leq 1,~\langle \grad g(\tilde{\bD}_k),\tilde{\bD}-\tilde{\bD}_k \rangle\leq  g(\tilde{\bD}_0)-g(\tilde{\bD}_k).
\end{equation}
By using the KKT condition, it can be shown that the above problem is equivalent to finding a zero of the following one-dimensional nonlinear equation involving $\lambda \geq 0$: 
\begin{equation*}
   \tilde{\bD} =  \Pi_{\mathcal{Z}}(\grad f_{\tilde{\bD}}(\tilde{\bD}_k,\tilde{\bX}_k)+\lambda \grad g(\tilde{\bD}_k)), \quad \langle \grad g(\tilde{\bD}_k),\tilde{\bD}-\tilde{\bD}_k \rangle=  g(\tilde{\bD}_0)-g(\tilde{\bD}_k),
\end{equation*}
where the projection on $\mathcal{Z}=\{\tilde{\bD}\in \reals^{25\times 40}:\|\tilde{\bd}_i\|_2\leq 1,\;i=1,\dots,40\}$ amounts to a column-wise projection on the Euclidean ball. In practice, we find that it can be solved efficiently by MATLAB's root-finding solver.

\blue{For BiG-SAM, we set $\eta_f=\eta_g=0.1$ and $\gamma =10$. For a-IRG, we set $\gamma_0=0.01$ and $\eta_0=1$. For DBGD, we set $\gamma_k = 0.1$.}
}

\end{document}